\documentclass[reqno]{amsart}

\usepackage[all]{xy}
\usepackage{graphicx}

\usepackage{amssymb}
\usepackage{amsmath}
\usepackage{mathrsfs}
\usepackage{epsfig}
\usepackage{amscd}
\usepackage{graphicx, color}

\def\E{\ifmmode{\mathbb E}\else{$\mathbb E$}\fi} 
\def\N{\ifmmode{\mathbb N}\else{$\mathbb N$}\fi} 
\def\R{\ifmmode{\mathbb R}\else{$\mathbb R$}\fi} 
\def\Q{\ifmmode{\mathbb Q}\else{$\mathbb Q$}\fi} 
\def\C{\ifmmode{\mathbb C}\else{$\mathbb C$}\fi} 
\def\H{\ifmmode{\mathbb H}\else{$\mathbb H$}\fi} 
\def\Z{\ifmmode{\mathbb Z}\else{$\mathbb Z$}\fi} 
\def\P{\ifmmode{\mathbb P}\else{$\mathbb P$}\fi} 
\def\T{\ifmmode{\mathbb T}\else{$\mathbb T$}\fi} 
\def\SS{\ifmmode{\mathbb S}\else{$\mathbb S$}\fi} 
\def\DD{\ifmmode{\mathbb D}\else{$\mathbb D$}\fi} 

\newcommand{\del}{\partial}
\newcommand{\Cont}{{\operatorname{Cont}}}

\newcommand{\ben}{\begin{enumerate}}
\newcommand{\een}{\end{enumerate}}
\newcommand{\be}{\begin{equation}}
\newcommand{\ee}{\end{equation}}
\newcommand{\bea}{\begin{eqnarray}}
\newcommand{\eea}{\end{eqnarray}}
\newcommand{\beastar}{\begin{eqnarray*}}
\newcommand{\eeastar}{\end{eqnarray*}}
\newcommand{\bc}{\begin{center}}
\newcommand{\ec}{\end{center}}

\theoremstyle{theorem}
\newtheorem{thm}{Theorem}[section]
\newtheorem{cor}[thm]{Corollary}
\newtheorem{lem}[thm]{Lemma}
\newtheorem{prop}[thm]{Proposition}

\theoremstyle{definition}
\newtheorem{defn}[thm]{Definition}
\newtheorem{rem}[thm]{Remark}

\newtheorem{exm}[thm]{Example}
\newtheorem{hypo}[thm]{Hypothesis}

\newtheorem*{thm*}{Theorem}

\numberwithin{equation}{section}

\def\R{{\mathbb R}}

\def\Crit{{\hbox{Crit}}}

\def\E{{\mathbb E}}
\def\Z{{\mathbb Z}}
\def\C{{\mathbb C}}
\def\R{{\mathbb R}}
\def\P{{\mathbb P}}

\def\N{{\mathbb N}}

\def\11{{\mathbb I}}

\def\delbar{{\overline \partial}}

\def\C{\mathbb{C}}
\def\Z{\mathbb{Z}}

\def\T{\mathbb{T}}

\def\Q{\mathbb{Q}}

\def\E{\ifmmode{\mathbb E}\else{$\mathbb E$}\fi} 
\def\N{\ifmmode{\mathbb N}\else{$\mathbb N$}\fi} 
\def\R{\ifmmode{\mathbb R}\else{$\mathbb R$}\fi} 
\def\Q{\ifmmode{\mathbb Q}\else{$\mathbb Q$}\fi} 
\def\C{\ifmmode{\mathbb C}\else{$\mathbb C$}\fi} 
\def\H{\ifmmode{\mathbb H}\else{$\mathbb H$}\fi} 
\def\Z{\ifmmode{\mathbb Z}\else{$\mathbb Z$}\fi} 
\def\P{\ifmmode{\mathbb P}\else{$\mathbb P$}\fi} 
\def\SS{\ifmmode{\mathbb S}\else{$\mathbb S$}\fi} 
\def\DD{\ifmmode{\mathbb D}\else{$\mathbb D$}\fi} 

\def\R{{\mathbb R}}

\def\Crit{{\hbox{Crit}}}
\def\E{{\mathbb E}}
\def\Z{{\mathbb Z}}
\def\C{{\mathbb C}}
\def\R{{\mathbb R}}

\def\N{{\mathbb N}}

\def\delbar{{\overline \partial}}

\def\CA{{\mathcal A}}

\def\CH{{\mathcal H}}

\def\CL{{\mathcal L}}

\def\CP{{\mathcal P}}

\def\CP{{\mathcal P}}

\def\darr#1{\raise1.5ex\hbox{$\leftrightarrow$}
\mkern-16.5mu #1}

\def\roughly#1{\raise.3ex\hbox{$#1$\kern-.75em
\lower1ex\hbox{$\sim$}}}

\def\opname#1{\mathop{\kern0pt{\rm #1}}\nolimits}

\def\dim{\opname{dim}}

\def\Dev{\operatorname{Dev}}

\def\span{\operatorname{span}}

\def\Cont{\operatorname{Cont}}
\def\Crit{\operatorname{Crit}}
\def\Spec{\operatorname{Spec}}
\def\Sing{\operatorname{Sing}}
\def\GFQI{\frak{G}}
\def\Index{\operatorname{Index}}

\def\Image{\operatorname{Image}}

\def\Int{\operatorname{Int}}

\begin{document}

\quad \vskip1.375truein

\def\mq{\mathfrak{q}}
\def\mp{\mathfrak{p}}
\def\mH{\mathfrak{H}}
\def\mh{\mathfrak{h}}
\def\ma{\mathfrak{a}}
\def\ms{\mathfrak{s}}
\def\mm{\mathfrak{m}}
\def\mn{\mathfrak{n}}
\def\mz{\mathfrak{z}}
\def\mw{\mathfrak{w}}
\def\Hoch{{\tt Hoch}}
\def\mt{\mathfrak{t}}
\def\ml{\mathfrak{l}}
\def\mT{\mathfrak{T}}
\def\mL{\mathfrak{L}}
\def\mg{\mathfrak{g}}
\def\md{\mathfrak{d}}
\def\mr{\mathfrak{r}}
\def\Cont{\operatorname{Cont}}
\def\Crit{\operatorname{Crit}}
\def\Spec{\operatorname{Spec}}
\def\Sing{\operatorname{Sing}}
\def\GFQI{\text{\rm g.f.q.i.}}
\def\Index{\operatorname{Index}}
\def\Cross{\operatorname{Cross}}
\def\Ham{\operatorname{Ham}}
\def\Fix{\operatorname{Fix}}
\def\Graph{\operatorname{Graph}}

\title[Contact Hamiltonian dynamics and contact instantons]
{Contact Hamiltonian dynamics \\
and perturbed contact instantons \\
with Legendrian boundary
condition}
\author{Yong-Geun Oh}
\address{Center for Geometry and Physics, Institute for Basic Science (IBS),
77 Cheongam-ro, Nam-gu, Pohang-si, Gyeongsangbuk-do, Korea 790-784
\& POSTECH, Gyeongsangbuk-do, Korea}
\email{yongoh1@postech.ac.kr}
\thanks{This work is supported by the IBS project \# IBS-R003-D1}


\begin{abstract}
This is the first of a series of papers in preparation in which we study the Hamiltonian perturbed
contact instantons with Legendrian boundary condition and its applications.
In this paper, we establish nonlinear ellipticity of this boundary value
problem by proving the a priori elliptic coercive estimates for the
contact instantons with Legendrian boundary condition, and prove an asymptotic
exponential $C^\infty$-convergence result at a puncture under the uniform $C^1$ bound.
We prove that the asymptotic charge of contact instantons  at the punctures
\emph{under the Legendrian boundary condition} vanishes,
which eliminates the phenomenon of the appearance of
\emph{spiraling cusp instanton along a Reeb core}.
This removes the only remaining obstacle towards the compactification and the Fredholm
theory of the moduli space of contact instantons in the open string case, which plagues the
closed string case.

In sequels to the present paper, we study the $C^1$ estimates by defining a proper notion of energy
for the contact instantons, and develop a Fredholm theory and construct a Gromov-type
compactification of the moduli space of contact instantons  with Legendrian boundary
condition and of finite energy, and apply them to problems in contact topology and dynamics.
\end{abstract}


\maketitle

\tableofcontents

\section{Introduction}

This is the first part of a series of papers in preparation in which we study the Hamiltonian perturbed
contact instantons with Legendrian boundary condition. The purpose of our study of this problem is two-fold.
The first one is to construct a Floer theoretic
construction of Legendrian spectral invariants on the one-jet bundle given by
Th\'eret in \cite{theret}, Bhupal \cite{bhupal} and Sandon \cite{sandon},
which are constructed using the generation functions quadratic at infinity ($\GFQI$).
This is the Legendrian version of Viterbo's $\GFQI$ spectral invariants
constructed in \cite{viterbo} for the Lagrangian submanifolds in the cotangent bundle.

Recall that the Floer theoretic construction of Viterbo's invariant is given
by the present author in \cite{oh:jdg,oh:cag}. The starting point of
this Floer theoretic construction was a remarkable
observation of Weinstein \cite{alan:observation} which reads that \emph{the classical action functional
is a generating function of
the time-one image $\phi_H^1(0_{T^*B})$ of the zero section $0_{T^*B}$ under the Hamiltonian flow of $H = H(t,x)$.}

Therefore the natural first step towards Floer theoretic construction of Legendrian spectral invariants
is to find a similar formulation
of the contact version of Weinstein's observation. More precisely, let
$$
\lambda = dz - pdq
$$
be the standard contact one-form on $J^1B$ and $R = \psi_H^1(0_{J^1B})$ is the
Legendrian submanifold which is the time-one image of the contact flow $\psi_H^t$
associated to the time-dependent function $H = H(t,y)$ with $y = (x,z) \in J^1B$.
We have found the contact counterpart of Weinstein's observation whose detailed
explanation will be given in a sequel \cite{oh-yso} to this paper.

Then towards the construction of a Legendrian counterpart of the construction in \cite{oh:jdg,oh:cag}
it turns out that one needs to generalize the notion of \emph{contact instantons}
introduced by Wang and present author \cite{oh-wang1,oh-wang3}. In a series of papers,
\cite{oh-wang2,oh-wang3} jointed with Wang and in \cite{oh:contacton}, the present author developed analysis of
contact Cauchy-Riemann maps \emph{without taking symplectization}. There is the phenomenon of the
\emph{appearance of spiraling contact instantons along the Reeb core}, even for finite
$\pi$-energy instantons, which is caused by
a puncture at which the asymptotic charge of a contact instanton is nonzero while the asymptotic
period is zero. The first main result of the present paper is a vanishing result of this charge
for the instanton \emph{with Legendrian boundary condition}, and
the asymptotic convergence theorem stated below in Theorem \ref{thm:subsequence-intro}.

Let $(\Sigma,j)$ be a compact Riemann surface with boundary and $\dot \Sigma$ a
punctured Riemann surface with a finite number of punctures which may be either interior
or boundary.

Then we extend the theory of contact instantons in two directions.
One is to establish the Legendrian boundary condition for the contact instanton equation
for a map $w: \dot \Sigma \to M$ with $\dot \Sigma$ from a
punctured bordered Riemann surface as an elliptic boundary condition
by establishing the a priori coercive elliptic estimates.
For a  $(k+1)$-tuple $(R_0,R_1, \cdots, R_k)$ of Legendrian submanifolds, we consider
the boundary value problem
\be\label{eq:contacton-Legendrian-bdy-intro}
\begin{cases}
\delbar^\pi w = 0, \quad d(w^*\lambda \circ j) = 0,\\
w(\overline{z_iz_{i+1}}) \subset R_i.
\end{cases}
\ee

The other is to introduce a Hamiltonian-perturbed contact instanton equation again as
an elliptic boundary value problem.
Consider a time-dependent function $H = H(t,y): \R \times M \to \R$. Denote by
$X_H$ the associated contact vector field. Then we have $\lambda(X_H) = -H$.
For each given coorientation preserving contact diffeomorphism $\psi$ of $(M,\xi)$, we have
$\psi^*\lambda = e^g \lambda$  for some function $g$ which we denote by $g=g_\psi$.

The following perturbed contact instanton equation is the contact counterpart of the celebrated Floer's Hamiltonian-perturbed
Cauchy-Riemann equation in symplectic geometry.

\begin{defn}\label{defn:contacton-Legendrian-bdy}
Let $(M,\lambda)$ be contact manifold equipped with a contact form, and
consider the (time-dependent) contact triad
$$
(M,\lambda, J), \quad J = \{J_t\}_{t \in [0,1]}.
$$
Let $H = H(t,x)$ be a time-dependent Hamiltonian.
We say $u: \R \times [0,1] \to M$ is a
\emph{$X_H$-perturbed Legendrian Floer trajectory} if it satisfies
\be\label{eq:perturbed-contacton-bdy-intro}
\begin{cases}
(du - X_H \otimes dt)^{\pi(0,1)} = 0, \quad d(e^{g_H(u)}(u^*\lambda + H\, dt)\circ j) = 0\\
u(\tau,0) \in R_0, \quad u(\tau,1) \in R_1
\end{cases}
\ee
where the function $g_H(u): \R\times [0,1] \to \R$ is defined by
\be\label{eq:gHu}
g_H(u)(t,x) := g_{((\psi_H^t (\psi_H^1)^{-1})^{-1}}(u(t,x)).
\ee
\end{defn}

The presence of the second equation in \eqref{eq:perturbed-contacton-bdy-intro} may
look somewhat mysterious but it turns out to be a natural
equation for the \emph{instanton} connecting two
contact Hamiltonian (Moore) trajectories: The contact Hamilton's equation can be
decomposed
\be\label{eq:equation-decompose-intro}
\dot x = X_H(t,x) \Longleftrightarrow
\begin{cases} (\dot x - X_H(t,x))^\pi = 0, \\
\gamma^*(\lambda + H\, dt) = 0,\\
\gamma(0) \in R_0, \, \gamma(1) \in R_1
\end{cases}
\ee
into the $\xi$-component and the Reeb component of the equation.
See Section \ref{sec:1st-variation} for more discussion on this.
\begin{rem} In physics literature (\cite{witten:morse} for example), the term \emph{instanton} is used to
mean a connecting gradient trajectory of two critical points of a function (or a functional). Regarding
 a contact Hamiltonian trajectory as a `critical point' of the contact action functional,
we would like to regard a contact instanton with two punctures as such a connecting
`gradient trajectory'.
\end{rem}

We can regard a contact Hamiltonian vector field $X_{-H}$ for a \emph{positive}
time-dependent Hamiltonian $H$ as the
time-dependent Reeb vector field of the time-dependent contact form $H^{-1}\lambda$ and then
\eqref{eq:perturbed-contacton-bdy-intro}
itself as the associated (unperturbed) time-dependent contact instanton equation.
In the present paper, we will focus on the analysis of unperturbed equation \eqref{eq:contacton-Legendrian-bdy-intro}
leaving necessary detailed discussion on the Hamiltonian-perturbed equation
\eqref{eq:perturbed-contacton-bdy-intro} and its application to \cite{oh:entanglement1}.
We also refer readers to \cite{oh:entanglement1} for the explanation on what
the choice of the conformal exponent function $g_H$ in \eqref{eq:gHu} is about.

As the first step towards the analytic study of the
above boundary value problem of the contact instanton,
we prove the following elliptic $W^{2,2}$-estimates in the general context of
maps $w: \dot \Sigma \to M$ for a Riemann surface
of genus zero with a finite number of boundary punctures.

Let $R_0, \ldots, R_k$ be a tuple of
Legendrian submanifolds and $\dot \Sigma = \Sigma \setminus\{z_0,\ldots z_k\}$
with $z_0,\ldots z_k \subset \del \Sigma$.
Let $w: \dot \Sigma \to M$ satisfy
\be\label{eq:contacton-Legendrian-bdy-intro}
\begin{cases}
\delbar^\pi w = 0 \, \quad d(w^*\lambda \circ j) = 0\\
w(z) \in R_i, \quad \text{\rm for } \, z \in \overline{z_iz_{i+1}}
\end{cases}
\ee

\begin{thm}\label{thm:local-W12-intro}
Let $w: \R \times [0,1] \to M$ satisfy \eqref{eq:contacton-Legendrian-bdy-intro}.
Then for any relatively compact domains $D_1$ and $D_2$ in
$\dot\Sigma$ such that $\overline{D_1}\subset D_2$, we have
$$
\|dw\|^2_{W^{1,2}(D_1)}\leq C_1 \|dw\|^2_{L^2(D_2)} + C_2 \|dw\|^4_{L^4(D_2)} + C_3 \|dw\|^3_{L^3(\del D_2)}
$$
where $C_1, \ C_2$ are some constants which
depend only on $D_1$, $D_2$ and $(M,\lambda, J)$ and $C_3$ is a
constant which also depends on $R_i$ as well.
\end{thm}

Starting from Theorem \ref{thm:local-W12-intro} and using the embedding $W^{2,2} \hookrightarrow C^{0,\alpha}$
with $0 < \alpha < 1/2$,
we also establish the following higher local $C^{k,\alpha}$-estimates on punctured surfaces $\dot \Sigma$
in terms of the $W^{2,2}$-norms with $\ell \leq k+1$.

\begin{thm}\label{thm:local-regularity-intro} Let $w$ satisfy \eqref{eq:contacton-Legendrian-bdy-intro}.
Then for any pair of domains $D_1 \subset D_2 \subset \dot \Sigma$ such that $\overline{D_1}\subset D_2$,
$$
\|dw\|_{C^{k,\alpha}} \leq C \|dw\|_{W^{1,2}(D_2)}
$$
where $C> 0$ depends on $J$, $\lambda$ and $D_1, \, D_2$ but independent of $w$.

In particular, any weak solution of \eqref{eq:contacton-Legendrian-bdy-intro} in
$W^{1,4}_{\text{\rm loc}}$ automatically becomes a classical solution.
\end{thm}

Next we study the asymptotic convergence result of contact instantons
of finite $E^\pi$-energy with Legendrian pair $(R_0,R_1)$ near the punctures of
a Riemann surface $\dot \Sigma$.

Let $\dot\Sigma$ be a punctured Riemann surface with punctures
$\{p^+_i\}_{i=1, \cdots, l^+}\cup \{p^-_j\}_{j=1, \cdots, l^-}$ equipped
with a metric $h$ with \emph{cylindrical ends} outside a compact subset $K_\Sigma$.
Let
$w: \dot \Sigma \to M$ be any smooth map. As in \cite{oh-wang2}, we define the total $\pi$-harmonic energy $E^\pi(w)$
by
\be\label{eq:endenergy}
E^\pi(w) = E^\pi_{(\lambda,J;\dot\Sigma,h)}(w) = \frac{1}{2} \int_{\dot \Sigma} |d^\pi w|^2
\ee
where the norm is taken in terms of the given metric $h$ on $\dot \Sigma$ and the triad metric on $M$.

We put the following hypotheses in our asymptotic study of the finite
energy contact instanton maps $w$ as in \cite{oh-wang2}:
\begin{hypo}\label{hypo:basic-intro}
Let $h$ be the metric on $\dot \Sigma$ given above.
Assume $w:\dot\Sigma\to M$ satisfies the contact instanton equations \eqref{eq:contacton-Legendrian-bdy-intro},
and
\begin{enumerate}
\item $E^\pi_{(\lambda,J;\dot\Sigma,h)}(w)<\infty$ (finite $\pi$-energy);
\item $\|d w\|_{C^0(\dot\Sigma)} <\infty$.
\item $\Image w \subset K \subset M$ for some compact set $K$.
\end{enumerate}
\end{hypo}
Let $w$ satisfy Hypothesis \ref{hypo:basic-intro}. We can associate two
natural asymptotic invariants at each puncture defined as
\bea
T & := & \frac{1}{2}\int_{[0,\infty) \times [0,1]} |d^\pi w|^2 + \int_{\{0\}\times [0,1]}(w|_{\{0\}\times [0,1]})^*\lambda\label{eq:TQ-T}\\
Q & : = & \int_{\{0\}\times [0,1]}((w|_{\{0\}\times })^*\lambda\circ j).\label{eq:TQ-Q}
\eea
(Here we only look at positive punctures. The case of negative punctures is similar.)
As in \cite{oh-wang2}, we call $T$ the \emph{asymptotic contact action}
and $Q$ the \emph{asymptotic contact charge} of the contact instanton $w$ at the given puncture.

\begin{thm}[Vanishing Charge and Subsequence Convergence]\label{thm:subsequence-intro}
Let $w:[0, \infty)\times [0,1]\to M$ satisfy the contact instanton equations \eqref{eq:contacton-Legendrian-bdy-intro}
and Hypothesis \ref{hypo:basic-intro}.
Then for any sequence $s_k\to \infty$, there exists a subsequence, still denoted by $s_k$, and a
massless instanton $w_\infty(\tau,t)$ (i.e., $E^\pi(w_\infty) = 0$)
on the cylinder $\R \times [0,1]$  such that
$$
\lim_{k\to \infty}w(s_k + \tau, t) = w_\infty(\tau,t)
$$
in the $C^l(K \times [0,1], M)$ sense for any $l$, where $K\subset [0,\infty)$ is an arbitrary compact set.
Furthermore, $w_\infty$ has $Q = 0$ and the formula $w_\infty(\tau,t)= \gamma(T\, t)$  with
asymptotic action $T$, where $\gamma$ is some Reeb chord joining $R_0$ and $R_1$ of period $|T|$.
\end{thm}

\begin{rem}\label{rem:big-difference-intro}
The vanishing $Q = 0$ of the asymptotic charge of
contact instanton $w$ with (nonempty) Legendrian boundaries is a big deviation from
the closed string case studied in \cite{oh-wang2,oh-wang3,oh:contacton}.
This vanishing removes the only remaining obstruction, the so called the \emph{appearance of spiraling instantons along the
Reeb core}, to a Fredholm theory and to a compactification of the moduli space of contact instantons.
This obstruction has been the main reason why the paper \cite{oh:contacton} has not been released in public
but being held in the author's homepage until recently since 2013.
\end{rem}

The coercive elliptic regularity estimates and the vanishing theorem of asymptotic charge in the present
work will be the foundation of the analytic package of Gromov-Floer-Hofer compactification and
the Fredholm theory of Hamiltonian-perturbed contact instantons with Legendrian boundary condition and
its applications. In \cite{oh:entanglement1} the author developed this analytic package for the
study of the moduli space of contact instantons perturbed by a (parameterized) cut-off Hamiltonian
of the type used in \cite{oh:mrl}. Using the analytic package,
we prove a conjecture of Shelukhin \cite{shelukhin:contactomorphism}, which concerns existence of translated points of
contactomorphisms \cite{sandon:translated}, as a consequence of its Legendrianization,
a general existence result of Reeb chords between the pair $(R,\psi(R))$ for any
compact Legendrian submanifolds.

In a sequel \cite{oh:perturbed-contacton-bdy}, we will develop the coercive
estimates for the perturbed contact instantons extending the current estimates
to the case of with non-zero Hamiltonians, and a relevant Fredholm theory elsewhere
by adapting the one from \cite{oh:contacton}
to the current case of contact instantons with boundary.
In other sequels, we will also apply these for the construction of
the Fukaya-type category on contact manifolds whose objects are Legendrian submanifolds
and morphisms and products which will be constructed by counting appropriate contact
instantons satisfying some prescribed asymptotic conditions and Legendrian boundary conditions.
 It will be interesting to see if Chekanov-Eliashberg-type DGA can
be constructed through such a study, which we will investigate elsewhere in a future.

In another sequel with Yu \cite{oh-yso}, utilizing these background geometric preparation,
we will give the Floer theoretic construction of Legendrian spectral invariants and provide
some applications thereof.

Now is the organization of the paper in order. In Section \ref{sec:contact-Hamiltonian}, we provide
a systematic exposition on the basic contact Hamiltonian geometry which will be needed for
the a priori elliptic estimates. Majority of the results in this section
are widely known to experts but are scattered around in literature with various different sign conventions.
Partly to fix our sign conventions, we organize them in the way that will be useful for our tensorial calculations.
In Section \ref{sec:1st-variation}, we compute the first variation of the contact action functional
and explain the relationship between its critical point and the contact Hamiltonian trajectories.
 In Section \ref{sec:contact-instantons} and \ref{sec:Wk+22-estimates},
we establish the local $W^{2,2}$ and $C^{k,\alpha}$ estimates for $k \geq 1$ with $0< \alpha< 1/2$ respectively
for the contact instantons with Legendrian boundary conditions. In Section \ref{sec:subsequence-convergence}
we prove that the presence of Legendrian barrier forces the asymptotic charge to vanish and
that for any contact instantons with derivative bound there exists a sequence $\tau_j \to \infty$
for which  $u(\tau_j, t) \to \gamma(T t)$ uniformly as $j\to \infty$ for some $T > 0$
in the strip-like coordinates $(\tau,t)$ near each puncture of the domain $\dot \Sigma$ of the map $u$.
In Section \ref{sec:exponential-convergence}, we prove that when the Legendrian boundary $(R_0,R_1)$
near a puncture
are transversal in a suitable sense then the above convergence is exponentially fast. In
\ref{sec:future-works}, we apply the contact instanton equation to the one-jet bundle $J^1B$ and
rewrite the equation in terms of the product coordinates $w = (u,f)$ of $J^1B = T^*B \times \R$
with respect to \emph{$CR$ almost complex structure lifted from the cotangent bundle $T^*B$},
which gives rise to an interesting elliptic system.

We thank Rui Wang for her previous collaboration on the contact instantons with
the closed string case \cite{oh-wang1,oh-wang2,oh-wang3}, and for her useful comments on the preliminary version of the
present paper.

\bigskip

\noindent{\bf Convention:}

\medskip

\begin{itemize}
\item {(Contact Hamiltonian)} The contact Hamiltonian of a time-dependent contact vector field $X_t$ is
given by
$$
H: = - \lambda(X_t).
$$
We denote by $X_H$ the contact vector field whose associated contact Hamiltonian is given by $H = H(t,x)$.
This convention is consistent with that of \cite{bhupal,BCT,dMV} but is the negative to that of \cite[Appendix 4]{arnold}
and \cite{mueller-spaeth}.
\item The Hamiltonian vector field on symplectic manifold $(P, \omega)$ is defined by $X_H \rfloor \omega = dH$.
We denote by
$$
\psi_H: t\mapsto \psi_H^t
$$
its Hamiltonian flow.
\item Contact action functional denoted by $\CA_H$ is given by
$$
\CA_H(\gamma) =  \int \gamma^*(-\lambda - H dt) = -\int \gamma^*\lambda - \int_0^1 H(t,\gamma(t))\, dt.
$$
The choice of the negative sign in front of $\int\gamma^*\lambda$ is
to be consistent with the classical action functional on the cotangent bundle $T^*B$ which is given by
$$
\CA_H(\gamma) = \int \gamma^*\theta- \int_0^1 H(t,\gamma(t))\, dt
$$
where $\lambda = dz - \theta$ is the canonical contact form on $J^1B$.
\item We denote by $R_\lambda$ the Reeb vector field associated to $\lambda$. We denote by
$$
\phi_{R_\lambda}^t
$$
its flow.
\end{itemize}

\section{Some contact Hamiltonian geometry and dynamics}
\label{sec:contact-Hamiltonian}

In this section, we state some useful results concerning the contact Hamiltonian dynamics.
Partly to fix our sign conventions, we
organize them in the way that will be useful for us to systematically perform geometric calculations
that enter in our elliptic estimates for contact instantons in the spirit of \cite{oh-wang2,oh-wang3}.
This systematic treatment will be needed for the analytical study of Hamiltonian-perturbed contact instanton
equations \eqref{eq:contacton-Legendrian-bdy-intro}. For this reason, we here include the contact Hamiltonian
calculi more than what we need for the purpose of the present paper (the case with $H = -1$), but
develop them as the contact Hamiltonian calculi similarly as in the Hamiltonian calculi in the symplectic case.
These will be needed in the sequels as in \cite{oh-wang4:contacton-Legendrian-bdy},
where we develop the coercive elliptic estimates for the perturbed equation \eqref{eq:contacton-Legendrian-bdy-intro}, and also applications thereof to contact Hamiltonian dynamics in \cite{oh-yso} and in others.

We also refer readers to \cite{BCT}, \cite{dMV} for some expositions on contact Hamiltonian
calculus (\emph{especially \cite{dMV} with the same sign conventions as ours
in both the symplectic and the contact contexts}), which partially overlap with
the exposition of the present section.

\subsection{Conformal exponents of contact diffeomorphisms}

We begin with the definition of contact diffeomorphisms.
\begin{defn} Let $(M,\xi)$ be a contact manifold. A diffeomorphism $\psi: M \to M$ is called a
contact diffeomorphism if $d\psi(\xi) \subset \xi$.
Denote by $\Cont(M,\xi)$ (resp. $\Cont_0(M,\xi)$)
the set of contact diffeomorphisms (resp. the identity component thereof).
\end{defn}

When $(M,\xi)$ is coorientable, we can choose a contact form
$\lambda$ with $\xi = \ker \lambda$. In the present paper, we will always assume
$(M,\xi)$ is cooriented and all contact forms represent the given coorientation.

\begin{rem} When $(M,\xi)$ is not coorientable, a natural treatment of such a contact manifold
is to regard it as a Jacobi manifold $(M,L,\{\cdot, \})$ where $L \to M$ is the line bundle
$L = TM/\xi$ and $\{\cdot, \cdot\}$ is a Lie bracket on the set of sections $\Gamma(L)$
which is a first order differential operator on each entry. See \cite{kirillov}, \cite[Definition 2.1]{LOTV}
for more details, and Proposition \ref{prop:brackets} coming later in the present paper for
the case when $(M,\xi)$ is cooriented.
\end{rem}

For given contact form $\lambda$, a coorientation preserving diffeomorphism $\psi$ of $(M,\xi)$
is contact if and only if it satisfies
$$
\psi^*\lambda = e^g \lambda
$$
for some smooth function $g: M \to \R$, which depends on the choice of contact form $\lambda$ of $\xi$.
Unless said otherwise, we will always assume that $\psi$ is coorientation preserving without mentioning
from now on.

\begin{defn} For given coorientation preserving contact diffeomorphism $\psi$ of $(M,\xi)$ we call
the function $g$ appearing in $\psi^*\lambda = e^g \lambda$ the \emph{conformal exponent}
for $\psi$ and denote it by $g=g_\psi$.
\end{defn}

The following lemma is a straightforward consequence
of the identity $(\phi\psi)^*\lambda = \psi^*\phi^*\lambda$.

\begin{lem}\label{lem:coboundary} Let $\lambda$ be given and denote by $g_\psi$ the function $g$
appearing above associated to $\psi$. Then
\begin{enumerate}
\item $g_{\phi\psi} = g_\phi\circ \psi + g_\psi$ for any $\phi, \, \psi \in \Cont(M,\xi)$,
\item $g_{\psi^{-1}} = - g_\psi \circ \psi^{-1}$ for all $\psi \in \Cont(M,\xi)$.
\end{enumerate}
\end{lem}

\begin{rem}\label{rem:coboundary} For each contact form $\lambda$ of the given contact manifold
$(Q,\xi)$, we consider the function $\varphi_\lambda: G:= \Cont(M,\xi) \to C^\infty(M)$,
defined by
\be\label{eq:varphi}
\varphi_\lambda(\psi) = g_\psi
\ee
if $\psi^*\lambda = e^{g_\psi}\lambda$. Regard it as a one-cochain in the group cohomology complex
$(C^1(G, C^\infty(M)), \delta)$ with the coboundary map $\delta: C^1(G, C^\infty(M)) \to C^2(G, C^\infty(M))$
on the right $\Z[G]$-module $C^1(G, C^\infty(M))$ with the right composition as the
action of $G$ on $C^1(G, C^\infty(M))$ $(\psi,g) \mapsto g\circ \psi$. Then the above lemma can be interpreted as
\be\label{eq:gphipsi}
g_{\phi\psi} = (\delta \varphi_\lambda)(\phi,\psi).
\ee
Furthermore for a different choice of contact form $\lambda'$ of the same contact structure $\xi$
in the same orientation class, we make $\lambda$-dependence on $g_\psi$ explicit by writing $g_\psi^\lambda$
and $g_\psi^{\lambda'}$. Furthermore we have $\lambda' = e^{h_{(\lambda'\lambda)}}\lambda$ for some
function $h_{(\lambda'\lambda)}$ depending on $\lambda, \, \lambda'$. Then a straightforward
calculation leads to $g_\psi^{\lambda'} = h_{\lambda'\lambda}\circ \psi + g_\psi^\lambda$. This itself
can be written as $\varphi_{\lambda'} - \varphi_{\lambda} = \delta h_{(\lambda'\lambda)}$ where
we regard $h_{(\lambda'\lambda)}$ as a zero-cochain in the group cohomology complex $C^*(G, C^\infty(M))$.
In other words, the two cochains $\varphi_{\lambda'},  \, \varphi_{\lambda}$ are cohomologus to each other.
\end{rem}

The following is an immediate corollary of Lemma \ref{lem:coboundary} which plays an important role in the study of
an invariance property of Legendrian spectral invariants of contact diffeomorphisms under the conjugation.
(See \cite{bhupal}, \cite{sandon}, for example.)

\begin{cor} Let $\lambda$ be any contact form of $(M,\xi)$ and consider the
conformal exponent map $\psi \mapsto g_\psi$ associated to $\lambda$. Then
\be\label{eq:phis*alpha}
g_{\phi\psi \phi^{-1}} = g_\phi \circ \psi \circ \phi^{-1} - g_\phi \circ \phi^{-1} + g_\psi \circ \phi^{-1}.
\ee
\end{cor}

Note that this formula is reduced to a simple one $g_{\phi\psi \phi^{-1}} = g_\psi \circ \phi^{-1}$ when
$\phi$ is a strict contactomorphism.
\subsection{Contact vector fields and their Hamiltonians}

\begin{defn} A vector field $X$ on $(M,\xi)$ is called \emph{contact} if $[X,\Gamma(\xi)] \subset \Gamma(\xi)$
where $\Gamma(\xi)$ is the space of sections of the vector bundle $\xi \to M$.
We denote by $\mathfrak X(M,\xi)$ the set of contact vector fields.
\end{defn}

When $\lambda$ is a contact form of $\xi$, it uniquely defines the Reeb vector field $R_\lambda$
by the defining condition
\be\label{eq:Rlambda}
R_\lambda \rfloor d\lambda = 0, \, \lambda(R_\lambda) = 1
\ee
and it admits a decomposition
\beastar
TM & = & \xi \oplus \R\langle R_\lambda \rangle\\
T^*M & = & (\R\langle R_\lambda \rangle)^\perp \oplus (\xi)^\perp
\eeastar
where $(\cdot)^\perp$ denotes the annihilator of $(\cdot)$.
Then the condition $[X,\Gamma(\xi)] \subset \Gamma(\xi)$ is
 equivalent to the condition that
there exists a smooth function $f: M \to \R$ such that
$$
\CL_X \lambda = f \lambda.
$$
\begin{defn} Let $\lambda$ be a contact form of $(M,\xi)$.
The associated function $H$ defined by
\be\label{eq:contact-Hamiltonian}
H = - \lambda(X)
\ee
is called the \emph{$\lambda$-contact Hamiltonian} of $X$. We also call $X$ the
\emph{$\lambda$-contact Hamiltonian vector field} associated to $H$.
\end{defn}
We alert readers that under our sign convention, the $\lambda$-Hamiltonian $H$ of the Reeb vector field $R_\lambda$
as a contact vector field becomes the constant function $H = -1$.

\begin{rem} Unlike the symplectic case, the vanishing locus of contact Hamiltonian $H$ has a
very natural geometric meaning: The locus is precisely the set of points at which the
contact vector field is tangent to the contact distribution $\xi$.
\end{rem}

Conversely, any smooth function $H$ associates a unique contact vector field
that satisfies the relationship spelled out in the above definition. We highlight
the fact that unlike the symplectic case, this correspondence is one-one with
no ambiguity of addition by constant. Existence and derivation of the formula of $X_H$ from $H$
is not as straightforward as that of Hamiltonian vector field in symplectic geometry.
Since we will mostly fix the contact form $\lambda$ in the present paper, we will simply
call $X_H$ a \emph{contact Hamiltonian vector field associated to $H$}
instead of $\lambda$-contact Hamiltonian vector field.

The following lemma is a special case of \cite[Lemma 2.1]{oh-wang3}. Since there is
a sign difference between the current paper
and \cite{oh-wang3} for the definitions of Hamiltonian vector fields and that of
contact Hamiltonian, we provide the full proof of the lemma for the readers'
convenience. Also see \cite[Section 2]{mueller-spaeth} for some relevant discussions.

We denote by $R_\lambda$ the Reeb vector field of $\lambda$.

\begin{prop} Let $H$ be any smooth function on $M$ equipped with contact form $\lambda$.
Then the contact Hamiltonian vector field $X_H$ is uniquely determined by
$H$, and satisfies the equation
\be\label{eq:dH}
dH = X_H \rfloor d\lambda  + R_\lambda[H]\, \lambda.
\ee
\end{prop}
\begin{proof} Recall the contact form $\lambda$ provides a canonical splitting
$$
TM = \xi \oplus \span_\R\{R_\lambda\}
$$
for the Reeb vector field $R_\lambda$. Therefore we can decompose
$$
X = X^\pi + X^\perp
$$
where $X^\pi$ (resp. $X^\perp$) is the component of $\xi$ (resp. of $\span_\R\{R_\lambda\}$).

The relation $H = - \lambda(X)$ already uniquely specifies the $R_\lambda$-component $X^\perp$
by
\be\label{eq:X-perp}
X^\perp = -H \, R_\lambda
\ee
since $\lambda(X^\pi) = 0$.
It remains to determine $\xi$-component $X^\pi$.
Thanks to the nondegeneracy of $d\lambda$ on $\xi$,
it is enough to determine $X^\pi \rfloor d\lambda$.
For this, we first take the differential of $H = - \lambda(X)$ and apply Cartan's formula to obtain
$$
dH = X \rfloor d\lambda - \CL_X \lambda.
$$
Then we note $X \rfloor d\lambda = X^\pi \rfloor d\lambda$ and that
$
\CL_X \lambda
$
is uniquely determined from the contact property of $X$
$$
\CL_X \lambda = f \lambda
$$
where $f$ is uniquely determined by $H$ as follows.
\begin{lem}\label{cor:LXtlambda} Let $X$ be a contact vector field with $\CL_X \lambda = f \lambda$,
and let $H$ be the associated contact Hamiltonian. Then $f = -R_\lambda[H]$.
\end{lem}
\begin{proof} Using the defining condition \eqref{eq:Rlambda}
of the Reeb vector field $R_\lambda$ of $\lambda$, and the formula $\lambda(X) = -H$,
we obtain
\be\label{eq:f=CLX}
f = (\CL_X \lambda) (R_\lambda) = - \lambda([X,R_\lambda]) = \CL_{R_\lambda}(\lambda(X))= - R_\lambda[H].
\ee
\end{proof}

Then $X^\pi \rfloor d\lambda$ is given by
\be\label{eq:X-parallel}
X^\pi \rfloor d\lambda = dH + \CL_X \lambda = dH - R_\lambda[H]\lambda.
\ee
Finally it follows from
\eqref{eq:X-perp} and \eqref{eq:X-parallel}, the right hand sides of which are uniquely determined by $H$, that $X$ satisfies \eqref{eq:dH}
which finishes the proof.
\end{proof}

In fact, the above proof proves an explicit formula for $X_H$
\be\label{eq:XH}
X_H = (\widetilde{d\lambda})^{-1}(dH)^{\perp_{R_\lambda}} - H R_\lambda
\ee
with $(dH)^{\perp_{R_\lambda}} =  dH - R_\lambda[H]\lambda$
where $\widetilde{d\lambda}: \xi \to \xi^* \cong \{R_\lambda\}^\perp \subset T^*M$ is the natural isomorphism
$Y \mapsto Y \rfloor d\lambda$ followed by the isomorphism $\xi^* \cong \{R_\lambda\}^\perp$
induced by the splitting $TM \cong \xi \oplus \span\{R_\lambda\}$. Here $\{R_\lambda\}^\perp$ is the annihilator of
$R_\lambda$. Note that the one form
$dH - R_\lambda[H]\lambda$ is indeed contained in $\{R_\lambda\}^\perp$.
In the canonical coordinates $(q,p,z)$ on $\R^{2n+1}$, this expression is reduced to
the well-known coordinate formula for the contact Hamiltonian vector field below.
 (See \cite[Appendix 4]{arnold}, \cite[Lemma 2.1]{oh-wang3}, \cite[Lemma 4.1]{bhupal}, \cite[(3.6)]{BCT} for example.)

\begin{exm} Following \cite[Section 6]{LOTV}, we introduce a vector field
$$
\frac{D}{\del q_i}: = \frac{\del}{\del q_i} + p_i \frac{\del}{\del z}
$$
which are tangent to $\xi$. Together with $\frac{\del}{\del p_i}$, the set
$
\left\{\frac{D}{\del q_i}, \frac{\del}{\del p_i}\right\}
$
provides a frame of $\xi$ on the chart of a canonical coordinates $(q_i,p_i,z)$.
(In \cite{LOTV}, $\frac{D}{\del q_i}$ is denoted by $D_i$ and there is a sign
difference in the choice of canonical symplectic form on $T^*B$.)

Let $H:\R^{2n+1} \to \R$ be a smooth function on $\R^{2n+1}$.
Then the contact Hamiltonian vector field $X_t = X_{H_t}$ is given by
\bea\label{eq:Xt-1}
X_t & =& \sum_{i=1}^n \left(\frac{\del H}{\del p_i} \frac{D}{\del q_i} - \frac{D H}{\del q_i} \frac{\del}{\del p_i}\right)
- H \frac{\del}{\del z} \nonumber \\
& = & \sum_{i=1}^n \frac{\del H}{\del p_i} \frac{\del}{\del q_i} -
\left(\frac{\del H}{\del q_i} + p_i \frac{\del H}{\del z}\right)\frac{\del}{\del p_i}
+ \left(\left\langle p, \frac{\del H}{\del p}\right \rangle  - H\right)\frac{\del}{\del z}
\eea
\end{exm}

\begin{rem} \begin{enumerate}
\item In the formula \cite[Lemma 2.1]{oh-wang3} applied to one-form $dH$
has no negative sign in front of $R_\lambda[H] \lambda$ of \eqref{eq:dH}. This is because
here and in \cite{BCT} the contact Hamiltonian is defined by $H = -\lambda(X)$ while in \cite{oh-wang3},
it is defined by $\lambda(X)$.
\item The formula \eqref{eq:Xt-1} explicitly shows that the assignment $H \mapsto X_H$ involves
a differential operator of mixed order, a combination of orders zero and one, unlike
the symplectic case. This difference makes contact Hamiltonian geometry and dynamics
exhibit quite different phenomena  from the symplectic case. This will be seen even in the study of
 the associated Hamilton's equation and the contact instanton equation.
\end{enumerate}
\end{rem}

\subsection{Contact vector field as a Reeb vector field}

Let $(M,\xi)$ be a contact manifold. For each given contact form $\lambda$,
any smooth function $H$ gives rise to a contact vector field, temporarily denoted as
$$
X_H= X_H^\lambda
$$
to emphasize the $\lambda$-dependence of the expression $X_H$. For another cooriented contact form $\lambda' = f \lambda$
with $f >0$, one might want to express $X_H^{\lambda'}$ in terms of $X_H^\lambda$. For general function $H$,
the relationship between them is rather complicated and does not seem to deserve writing it down. But for the
or the $f\lambda$-Reeb vector field $R_{f\lambda}$, which corresponds to $H = -1$,
we have the following.

\begin{prop}\label{prop:XH=Xf} Let $f = f(x)$ be a positive function on $(M,\lambda)$. Consider its inverse
$G = - 1/f$ and consider a new contact form $f \lambda$. Then we have
$$
R_{f\lambda} = X_G^\lambda \left(= \frac{1}{f} R_{\lambda} + X_{-1/f}^\pi\right).
$$
\end{prop}
\begin{proof} For the notational simplicity, we write $X_G = X_G^\lambda$ omitting the
superscript $\lambda$ in the following discussion.

Let $X_G^\pi$ be the $\xi$-projection of $X_G$. Then we have
$$
X_G = X_G^\pi - G R_\lambda \in \xi \oplus \R\langle R_\lambda \rangle.
$$
We evaluate
$$
\lambda(R_{f\lambda}) = \frac1f (f\lambda) (R_{f\lambda}) = 1/f =-G
$$
which shows that $X_G$ and $R_{f\lambda}$ have the same Reeb component.
We can uniquely express $R_{f\lambda}$ as
$$
R_{f\lambda} = \frac1f(R_\lambda + \eta)
$$
for some $\eta \in \xi$. Then we compute
\beastar
fR_{f\lambda} \rfloor d\lambda
&=& R_{f\lambda}\rfloor d(f\lambda)-R_{f\lambda}\rfloor(df\wedge\lambda)\\
&=&-R_{f\lambda}\rfloor(df\wedge\lambda)\\
&=&-R_{f\lambda}[f]\, \lambda+\lambda(R_{f\lambda})\, df\\
&=&-\frac{1}{f}(R_\lambda+\eta)[f]\, \lambda+\frac{1}{f}\lambda(R_\lambda+\eta)\, df\\
&=&-\frac{1}{f}R_\lambda[f]\, \lambda-\frac{1}{f}\eta[f]\, \lambda+\frac{1}{f}df.
\eeastar
Take value of $R_\lambda$ for both sides, we derive that $\eta$ satisfies
$
\eta[f]=0.
$
Therefore we have
$$
R_{f\lambda} \rfloor d\lambda= \frac{df}{f^2} -\frac{1}{f^2} R_\lambda[f] \,\lambda = - d(1/f) + R_\lambda[1/f]\, \lambda
= - X_{1/f} \rfloor d\lambda.
$$
This proves $(X_{-1/f})^\pi = (R_{f\lambda})^\pi$. Combining this with $\lambda(X_{1/f}) = - R_{f\lambda}$,
we have finished the proof.
\end{proof}

In this regard, for a given contact manifold $(M,\xi)$ equipped with a contact form $\lambda$, Reeb vector fields for
different contact form $f \lambda$ correspond to the special $\lambda$-contact vector fields whose
Hamiltonian $G = -1/f$ is nowhere vanishing.

\subsection{Contact isotopy, conformal exponent and contact Hamiltonian}

Next let $\psi_t$ be a contact isotopy of $(M, \xi = \ker \lambda)$ with $\psi_t^*\lambda = e^{g_t}\lambda$
and let $X_t$ be the time-dependent vector field generating the isotopy. Let
$H: [0,1] \times \R^{2n+1} \to \R$ be the associated time-dependent contact Hamiltonian
$H_t = - \lambda(X_t)$. One natural question to ask is explicitly what the relationship
between the Hamiltonian $H = H(t,x)$ and the conformal exponent $g = g(t,x)$ of $\Psi$ is.

Similarly as in the symplectic case \cite{oh:hameo1}, we first name the parametric assignment $\psi_t \to -\lambda(X_t)$
of $\lambda$-Hamiltonians the \emph{$\lambda$-developing map} $\Dev_\lambda$. We will just call it the developing map
when there is no need to highlight the $\lambda$-dependence of the contact Hamiltonian.

\begin{defn}[Developing map]\label{defn:Dev-lambda} Let $T \in \R$  be given. Denote by
$$
\CP_0([0,T], \Cont(M,\xi))
$$
the set of contact isotopies $\Psi= \{\psi_t\}_{t \in [0,T]}$ with $\psi_0 = id$.
We define the $\lambda$-developing map
$$
\Dev_\lambda: \CP([0,T], \Cont(M,\xi)) \to C^\infty([0,T] \times M)
$$
by the timewise assignment of Hamiltonians
$$
\Dev_\lambda(\Psi) : = -\lambda(X), \quad \lambda(X)(t,x) := \lambda(X_t(x)).
$$
\end{defn}

We also state the following as a part of contact Hamiltonian calculus, whose proof is
a straightforward calculation.

\begin{lem}\label{lem:inverse-Hamiltonian} Let $\Psi = \{\psi_t\} \in \CP([0,T],\Cont(M,\xi))$ be a contact isotopy satisfying
$\psi_t^*\lambda = e^{g_t} \lambda$ with $g_\Psi: = g(t,x)$and generated by the vector field $X_t$ with
its contact Hamiltonian  $H(t,x) = H_t(x)$, i.e., with $\Dev_\lambda(\Psi) = H$.
\begin{enumerate}
\item Then the (timewise) inverse isotopy $\Psi^{-1}: = \{\psi_t^{-1}\}$ is generated
by the (time-dependent) contact Hamiltonian
\be\label{eq:inverse-Hamiltonian}
\Dev_\lambda(\Psi^{-1}) = - e^{-g_\Psi \circ \Psi} \Dev_\lambda(\Psi)
\ee
where the function $g_\Psi \circ \Psi$ is given by $(g_\Psi \circ \Psi)(t,x) : = g_\Psi(t,\psi_t(x))$.
\item If $\Psi' = \{\psi'_t\}$ is another contact isotopy with conformal exponent $g_{\Psi'} = \{g'_t\}$,
then the timewise product $\Psi' \Psi$ is generated by the Hamiltonian
\be\label{eq:product-Hamiltonian}
\Dev_\lambda(\Psi'\Psi) = \Dev_\lambda(\Psi') + e^{-g_\Psi \circ (\Psi')^{-1}} \Dev_\lambda(\Psi) \circ (\Psi')^{-1}.
\ee
\end{enumerate}
\end{lem}
In particular, with the more standard notation as in \cite{mueller-spaeth}, we have
\be\label{eq:barH'H}
\overline{H}' \# H(t,x)(:=\Dev_\lambda(\Psi'^{-1}\Psi)) = e^{-g_t'(\psi_t')}(H - H')(t,\psi'_t(x))
\ee
when $H = \Dev_\lambda(\Psi)$ and $H' = \Dev_\lambda(\Psi')$ and $g_{\Psi'}(t,x) = g_t'(x)$.
We remark that these formulae are reduced to the standard formulae
in the Hamiltonian dynamics in symplectic geometry if $g_t \equiv 1 \equiv g'_t$,
i.e., if $\psi_t, \, \psi_t'$ are $\lambda$-strict contactomorphism.

Finally the following formula provides an explicit relationship between the contact Hamiltonians
and the conformal exponents of a given contact isotopy.

\begin{prop}
Let $\Psi=\{\psi_t\}$ be a contact isotopy of $(M, \xi = \ker \lambda)$ with $\psi_t^*\lambda = e^{g_t}\lambda$
with $\Dev_\lambda(\Psi) = H$. We write $g_\Psi(t,x) = g_t(x)$. Then
\be\label{eq:dtdg}
\frac{\del g_\Psi}{\del t}(t,x) = -  R_\lambda[H](t,\psi_t(x)).
\ee
In particular, if $\psi_0 = id$,
\be\label{eq:gp}
g_\Psi(t,x) = \int_0^t - R_\lambda[H] (u,\psi_u(x))\, du.
\ee
\end{prop}
\begin{proof} Clearly the latter equality at $t = 0$ holds since $\psi_0$ is the identity
and so $g_0 \equiv 0$. We differentiate $\psi_t^*\lambda = e^{g_t}\lambda$ to obtain
$$
\psi_t^*(\CL_{X_t}\lambda) = e^{g_t} \frac{\del g_t}{\del t} \lambda.
$$
But by Corollary \ref{cor:LXtlambda}, we have
\beastar
\psi_t^*(\CL_{X_t}\lambda)(t,x) & = &  \psi_t^*(-R_\lambda[H] \, \lambda)\\
& = & -(R_\lambda[H] \circ \psi_t)\,\psi_t^*\lambda = -(R_\lambda[H] \circ \psi_t)\,  e^{g_t} \lambda.
\eeastar
Comparing the two sides, we have derived
$$
\frac{\del g_\Psi}{\del t}(x) = -  R_\lambda[H](t,\psi_t(x))
$$
which is \eqref{eq:dtdg}.
By integrating it from 0 to $t$, we have derived \eqref{eq:gp}.
\end{proof}

\begin{rem} Bhupal \cite{bhupal} derived the corresponding formula \eqref{eq:gp} for the contact structure
$\lambda = dz - pdq$ in $\R^{2n+1}$
by a coordinate calculation. Since any contact structure admits canonical coordinates, his coordinate
calculation also derives \eqref{eq:gp}. Here we provide a simple coordinate-free proof by utilizing the basic
properties of contact Hamiltonian vector fields spelled out in this section.
\end{rem}

\subsection{Miscellaneous contact Hamiltonian calculi}

We also mention a few more interesting contact calculi, although they
will not be used in the present paper. These calculi are scattered around in various literature
with different contexts. (See \cite{LOTV}, \cite{dMV}, \cite{mueller-spaeth} for example.)

As a special case of Jacobi manifold $(M,L,\{\cdot,\cdot\})$
\cite{kirillov}, each contact manifold carries a natural Jacobi bracket
$$
\{\cdot, \cdot\}: \Gamma(L) \times \Gamma(L) \to \Gamma(L)
$$
which is the natural Lie bracket on $\mathfrak X(M,\xi)\cong T_{id}\Cont(M,\xi)$ of
the Lie group $\Cont(M,\xi)$. (See \cite{kirillov}, \cite{LOTV}.)

When $(M,\xi)$ is cooriented, the line bundle $L = \R_M$ is trivial and the associated
Jacobi bracket is also called the \emph{Lagrange bracket} in some literature. (See \cite{arnold-givental}.)
When a contact form $\lambda$ is given
the map $X \mapsto -\lambda(X)$ induces a Lie algebra isomorphism
given as follows. (Its existence is well-known, e.g., stated explicitly in \cite{arnold}, \cite{LOTV},
\cite{dMV} in a slightly different form but equivalent up to sign convention.)
Partly due to different sign conventions
literature-wise, we state them here to fix the sign convention of ours.

\begin{prop}[Compare with Proposition 5.6 \cite{LOTV}, Proposition 9 \cite{dMV}]
Let $(M,\xi)$ be cooriented and $\lambda$ be an associated contact form.
Define a bilinear map
$$
\{\cdot, \cdot\}: C^\infty(M) \times C^\infty(M) \to C^\infty(M).
$$
by
\be\label{eq:bracket}
\{H,G\}: = -\lambda([X_H,X_G]).
\ee
Then it satisfies Jacobi identity and the assignment $X \mapsto -\lambda(X)$ defines a
Lie algebra isomorphism $\mathfrak X(M,\xi)$ to $C^\infty(M)$.
\end{prop}

We state the following special cases explicitly  for a future purpose.

\begin{prop}\label{prop:brackets} Let $H: M \to \R$ be a function and $X_H$ the contact
Hamiltonian vector field. Then
\begin{enumerate}
\item $dH(X_H) = -H R_\lambda[H]$.
\item Let $X$ be a contact vector field and let $H$ be
its Hamiltonian. Then
$[R_\lambda, X_H] = X_{-R_\lambda[H]}$. In other words, we have $\{1, H\} = - R_\lambda[H]$.
\end{enumerate}
\end{prop}
\begin{proof}

For the identity (1), we just evaluate $dH$ against the formula \eqref{eq:Xt-1} of $X_H$ in canonical coordinates
for the time-independent $H$.

For the identity (2), we first note that the Hamiltonian of $R_\lambda$ as a contact
vector field  is the constant function $-1$.
We need to compute $\lambda([R_\lambda,X_H])$. We have
$$
\lambda([R_\lambda,X_H]) = R_\lambda[\lambda(X_H)] - (\CL_{R_\lambda}\lambda)(X_H)
= R_\lambda[- H] = - R_\lambda[H]
$$
which finishes the proof of the first statement. For the second, we recall $\lambda(R_\lambda) \equiv 1$
and hence $R_\lambda = X_{-1}$. Therefore we have
$$
\{1,H\} = -\lambda([X_1,X_H]) = - \lambda([- R_\lambda,X_H]) =\lambda([R_\lambda,X_H]) =- R_\lambda[H]
$$
where the last equality comes from the first statement.
\end{proof}

\begin{rem} \begin{enumerate}
\item The identity (1) above is interpreted as `energy dissipation' in \cite{dMV}.
One interesting consequence of this identity is that the contact Hamiltonian vector field $X_H$
is tangent to the zero-level set $H^{-1}(0)$. In this sense, the zero-level set forms an \emph{equilibrium status}
of the flow of $X_H$. The same also holds for the zero-level set of the function $R_\lambda[H]$.
\item The non-vanishing of the bracket $\{1,H\} = -R_\lambda[H]$
exhibits quite a contrast from the case of
symplectic Poisson bracket, which also indicates that the Jacobi bracket does not satisfy
the Leibniz rule.
\item The equality $\psi^*X_H = X_{H\circ \psi}$ holds for any strict contactomorphism $\psi$
not just for the Reeb flow. In other words, we have
\be\label{eq:pull-back}
\Dev_\lambda(\psi^{-1} \circ \Psi \circ \psi) = \Dev_\lambda(\Psi) \circ \psi
\ee
for any strict contactomorphism $\psi$. (See  \cite[Lemma 2.2]{mueller-spaeth}.)
\end{enumerate}
\end{rem}

\section{First variation of the contact action functional}
\label{sec:1st-variation}

Let $X_t$ be a contact Hamiltonian vector field and its associated Hamiltonian
$$
H(t,x) = - \lambda(X_t).
$$
We denote by $\psi_H^t$ its flow and $(\psi_H^t)^*\lambda = e^{g_t}\lambda$.
Consider the free path space
$$
\CP: = C^\infty([0,1],M) = \{\gamma: [0,1] \to M \}
$$
and consider the action functional
\be\label{eq:AAH}
\CA_H(\gamma) = -\int \gamma^*\lambda - \int_0^1 H(t,\gamma(t))\, dt
\ee
A straightforward calculation derives the following first variation formula
of the action functional \eqref{eq:AAH} on a \emph{free} path space $\CP$.

\begin{prop} For any vector field $\eta$ along $\gamma$, we have
\bea\label{eq:1st-variation}
\delta \CA_H(\gamma)(\eta) & = & - \int_0^1 \left(d\lambda(\dot \gamma - X_H(t, \gamma(t)), \eta) -
 R_\lambda[H](\gamma) \lambda(\eta)\right)\, dt \nonumber\\
 &{}& \quad  - \langle \lambda(\gamma(1)), \eta(\gamma(1)) \rangle
 + \langle \lambda(\gamma(0)), \eta(\gamma(0)) \rangle -a(1) + a(0).
\eea
\end{prop}
\begin{proof} Let $\gamma_s$ with $ -\varepsilon < s < \varepsilon$ such that
$$
\gamma_0 = \gamma, \quad \frac{\del \gamma_s}{\del s}\Big|_{s = 0} = \eta
$$
and compute
$$
\frac{d}{ds}\Big|_{s=0} \CA(\gamma_s) = - \frac{d}{ds}\Big|_{s=0}\int \gamma_s^*\lambda -
\frac{d}{ds}\Big|_{s=0} \int_0^1 H(t, \gamma_s(t))\, dt.
$$
For the first term, we have
$$
\frac{d}{ds}\Big|_{s=0} \gamma_s^*\lambda = d(\eta \rfloor \lambda) + \eta \rfloor d\lambda.
$$
On the other hand,
$$
dH = X_H \rfloor d\lambda + R_\lambda[H]\, \lambda.
$$
Combining the above, we have finished the proof.
\end{proof}

We recall that $\lambda$ provides a natural decomposition
$$
TM = \xi \oplus \R \langle R_\lambda \rangle
$$
which induces a decomposition $\eta(t) = \eta^\pi(t) + a(t) R_\lambda(\gamma(t))$
with $a(t) = \lambda(\eta(t))$ and hence
the splitting
$$
T_\gamma \CP =  \Gamma(\gamma^*TM) = \Gamma(\gamma^*\xi) \oplus C^\infty([0,1], \R)\langle R_\lambda \rangle.
$$
This in turn provides a decomposition of the first variation
$$
\delta \CA_H(\eta) = \delta \CA_H (\eta^\pi) + \delta \CA_H (a R_\lambda)
$$
In terms of this decomposition, we can rewrite \eqref{eq:1st-variation} into
\bea \label{eq:first-variation-AAH}
\delta \CA_H(\eta) & = & - \int_0^1 d\lambda((\dot \gamma - X_H(t, \gamma(t)))^\pi, \eta) \, dt
- \int_0^1 R_\lambda[H](\gamma) a(t)\, dt
\nonumber \\
 &{}& - \langle \lambda(\gamma(1)), \eta^\pi(\gamma(1)) \rangle
 + \langle \lambda(\gamma(0)), \eta^\pi(\gamma(0)) \rangle
 - a(1) + a(0)
\eea
for $\eta = \eta^\pi + a R_\lambda$. More precisely, we have
\bea
\delta \CA_H (\eta^\pi) & = & - \int_0^1 d\lambda((\dot \gamma - X_H(t, \gamma(t)))^\pi, \eta) \, dt
\nonumber\\
&{}& - \langle \lambda(\gamma(1)), \eta^\pi(\gamma(1)) \rangle
 + \langle \lambda(\gamma(0)), \eta^\pi(\gamma(0)) \rangle \label{eq:AAH=pi}\\
\delta \CA_H (a R_\lambda) & = & -\int_0^1 R_\lambda[H](\gamma) a(t)\, dt - a(1) + a(0). \label{eq:AAH-perp}
\eea

Summarizing the above discussion and considering the (Moore) path space starting from a
Legendrian submanifold $R_0$ with free end at $T$
$$
\CL(M;R_0) =\bigcup_{T \in \R} \{T\} \times \{\gamma:[0,T] \to M \mid \gamma(0) \in R_0\}
$$
we have proved that the $\pi$-critical point, i.e., a path
satisfying $\delta \CA_H(\gamma) = 0$ \emph{with fixed $T$} satisfies
\be\label{eq:Crit-AAH}
(\dot \gamma - X_H(t, \gamma(t)))^\pi = 0, \quad \gamma(0) \in R_0.
\ee
Then taking the $R_\lambda$ component  \eqref{eq:AAH-perp}
(with $1$ replaced by $T$), we obtain
\be\label{eq:a-Crit-AAH}
\int_0^T R_\lambda[H](\gamma) a(t)\, dt + a(T) = 0.
\ee
\begin{prop}\label{prop:constraint-variation} Consider the restriction of $\CA_H$ on
the level set curves $\gamma$ satisfying
\be\label{eq:Reeb-constraint}
\CA_H^{-1}(0) \cap \CL(M;R_0).
\ee
If $(\dot \gamma - X_H(t,\gamma))^\pi = 0$, $\gamma$ is
a critical point thereof in the level set $\CA_H^{-1}(0) \subset \CL(M;R_0)$
of codimension 1.
\end{prop}
\begin{proof} Let $\eta$ be a first variation of $\gamma$ under the variational
constraint \eqref{eq:Reeb-constraint}. Then any such $\eta = \eta^\pi + a R_\lambda$ satisfies
\beastar
0 & = & \delta \CA_H(\gamma)(\eta) \nonumber \\
& = & - \int_0^1 d\lambda((\dot \gamma - X_H(t, \gamma(t)))^\pi, \eta) \, dt
- \langle(\lambda(\gamma(1)), \eta(1) \rangle \nonumber\\
&{}& \quad -\int_0^1 R_\lambda[H](\gamma) a(t)\, dt - a(1) \label{eq:AAH-perp}
\eeastar
where we use the Legendrian boundary condition $\gamma(0) \in R_0$ applied to \eqref{eq:first-variation-AAH}.
In particular if $(\dot \gamma - X_H(\gamma))^\pi) = 0$, then we obtain
$$
0 = - \langle(\lambda(\gamma(1)), \eta(1) \rangle - \int_0^1 R_\lambda[H](\gamma) a(t)\, dt - a(1).
$$
By substituting these into \eqref{eq:1st-variation}, we have obtained
$$
\delta \CA_h(\gamma) = 0
$$
inside the level set $\CA_H^{-1}(0)$. This finishes the proof.
\end{proof}

\begin{rem} \begin{enumerate}
\item
In fact, we can decompose contact Hamilton's equation $\dot x = X_H(t,x)$ into
\be\label{eq:equation-decompose}
\dot x = X_H(t,x) \Longleftrightarrow
\begin{cases} (\dot x - X_H(t,x))^\pi = 0\\
\gamma^*(\lambda + H\, dt) = 0
\end{cases}
\ee
with respect to the decomposition of the tangent bundle $TM = \xi \oplus \R\langle R_\lambda \rangle$.
This explains somewhat mysterious appearance of the Reeb component $d(w^*\lambda \circ j) = 0$
for the contact instanton equation, the one
whose asymptotic limiting equation is the Reeb component $\gamma^*(\lambda + H\, dt)$,
contact Hamilton's equation.
\item Since any contact Hamiltonian trajectory satisfies $\gamma^*(\lambda + H\, dt) = 0$,
we have $\CA_H(\gamma) = 0$ for any
contact Hamiltonian trajectory $\gamma$. Therefore Proposition \ref{prop:constraint-variation}
shows that any contact Hamiltonian trajectory is a critical point of the aforementioned constrained variational problem
in Proposition \ref{prop:constraint-variation}.
\item
For general contact manifolds, there does not seem to exist a $0$-order constraint
for $\gamma$ at $t=T$ that naturally provides a gradient structure for the contact instanton equation.
In Section \ref{sec:future-works}, we will show that there is such a boundary condition
for the case of one-jet bundles.
\end{enumerate}
\end{rem}

We now examine the relationship between the critical points of
the aforementioned constrained action functional and the contact Hamiltonian trajectories.

\begin{prop}\label{prop:lifting} Suppose a path $\gamma:[0,1] \to M$ satisfies
\be\label{eq:Crit-AAH}
(\dot \gamma - X_H(t, \gamma(t)))^\pi = 0.
\ee
Define the function $\rho: [0,1] \to \R$ by $\rho(t): = -\CA_H(\gamma|_{[0,t]})$ and
consider the Reeb-translated Hamiltonian
$$
\widetilde H(t,x): = H(t, \phi_{R_\lambda}^{\rho(t)}(x)).
$$
Then the Reeb-translated curve
$$
\widetilde \gamma(t) = \phi_{R_\lambda}^{-\rho(t)}(\gamma(t))
$$
satisfies the Hamilton's equation $\dot x = X_{\widetilde H}(t,x)$ for $\widetilde H$.
\end{prop}
\begin{proof} Since $\gamma$ satisfies \eqref{eq:Crit-AAH}, we can write
$\dot \gamma(t) - X_H(t,\gamma(t)) = b(t) R_\lambda(\gamma(t))$, i.e.,
$$
\dot \gamma(t) = b(t) R_\lambda(\gamma(t)) + X_H(t,\gamma(t))
$$
for some function $b = b(t)$. In fact, we have
$$
b(t) = \lambda(\dot \gamma(t) - X_H(t,\gamma(t))) = \lambda(\dot \gamma(t)) + H(t,\gamma(t)).
$$

Now consider the flow of the time-dependent
vector field $b(t) R_\lambda$ which is just a reparameterization of the Reeb flow
$$
t \mapsto \phi_{R_\lambda}^{\rho(t)}, \quad \rho(t): = \int_0^t b(u)\, du = -\CA_H(\gamma|_{[0,t]}).
$$
By definition, we have $b(t) = \rho'(t)$. We define
$$
\widetilde \gamma(t): = \phi_{R_\lambda}^{-\rho(t)}(\gamma(t))
$$
and compute its derivative
\beastar
\frac{d}{dt} \widetilde \gamma(t) & = & - \rho'(t) R_\lambda(\widetilde \gamma(t))
+ d\psi_{R_{\lambda}}^{-\rho(t)} \dot \gamma(t)\\
& = & - \rho'(t) R_\lambda(\widetilde \gamma(t))
+ d\psi_{R_{\lambda}}^{\rho(t)}(\rho'(t)R_\lambda(\gamma(t)) + X_H(t,\gamma(t)))\\
& = & d\psi_{R_{\lambda}}^{-\rho(t)}(X_H(t,\gamma(t)))
= (\psi_{R_{\lambda}}^{\rho(t)})^*X_H(t,\widetilde \gamma(t))) \\
& = & X_{H_\circ \phi_{R_\lambda}^{\rho(t)}}(t,\widetilde \gamma(t))).
\eeastar
where we use the identity
$$
R_\lambda(\widetilde \gamma(t)) = d\psi_{R_{\lambda}}^{\rho(t)} R_\lambda(\gamma(t))
$$
for the penultimate equality and \eqref{eq:pull-back} for the last equality, respectively.
This finishes the proof.
\end{proof}

\begin{rem} In particular, the asymptotic limit $\gamma$ of any finite $\pi$-energy
solution of \eqref{eq:perturbed-contacton-bdy-intro} satisfies
$$
(\dot \gamma - X_H(t,\gamma(t)))^{\pi} = 0.
$$
It is shown in \cite{oh:entanglement1} that the asymptotic limit indeed of the form
$$
t \mapsto \phi_H^t \circ (\phi_H^1)^{-1}\circ \phi_{R_\lambda}^t(\phi_H^1(p))
$$
for some point $p \in R_0$.
\end{rem}

\section{Contact instantons with Legendrian boundary condition}
\label{sec:contact-instantons}

In this section, consider
the Hamiltonian-perturbed contact instanton equation \eqref{eq:contacton-Legendrian-bdy-intro}
with the finite $\pi$-energy condition
$$
\int_{\dot \Sigma} |(d w - X_H \otimes dt)^\pi|^2 < \infty
$$
on general contact manifolds $(M,\lambda)$ equipped with a $\lambda$-adapted
CR-almost complex structure called a \emph{contact triad} in \cite{oh-wang1,oh-wang2,oh-wang3}.

To highlight the main points of the a priori estimates of perturbed contact instantons \emph{with
Legendrian boundary condition}, we will restrict to the case with $H = -1$ in the rest of the present
paper postponing the discussion of the modifications needed to handle the Hamiltonian term to
\cite{oh-yso} and other sequels.
We start with a review of the contact triad connection introduced in \cite{oh-wang1}.

\subsection{Review of the contact triad connection}
\label{subsec:connection}

Assume $(M, \lambda, J)$ is a contact triad for the contact manifold $(M, \xi)$, and equip with it the contact triad metric
$g=g_\xi+\lambda\otimes\lambda$.
In \cite{oh-wang1}, the authors introduced the \emph{contact triad connection} associated to
every contact triad $(M, \lambda, J)$ with the contact triad metric and proved its existence and uniqueness.

\begin{thm}[Contact Triad Connection \cite{oh-wang1}]\label{thm:connection}
For every contact triad $(M,\lambda,J)$, there exists a unique affine connection $\nabla$, called the contact triad connection,
 satisfying the following properties:
\begin{enumerate}
\item The connection $\nabla$ is  metric with respect to the contact triad metric, i.e., $\nabla g=0$;
\item The torsion tensor $T$ of $\nabla$ satisfies $T(R_\lambda, \cdot)=0$;
\item The covariant derivatives satisfy $\nabla_{R_\lambda} R_\lambda = 0$, and $\nabla_Y R_\lambda\in \xi$ for any $Y\in \xi$;
\item The projection $\nabla^\pi := \pi \nabla|_\xi$ defines a Hermitian connection of the vector bundle
$\xi \to M$ with Hermitian structure $(d\lambda|_\xi, J)$;
\item The $\xi$-projection of the torsion $T$, denoted by $T^\pi: = \pi T$ satisfies the following property:
\be\label{eq:TJYYxi}
T^\pi(JY,Y) = 0
\ee
for all $Y$ tangent to $\xi$;
\item For $Y\in \xi$, we have the following
$$
\del^\nabla_Y R_\lambda:= \frac12(\nabla_Y R_\lambda- J\nabla_{JY} R_\lambda)=0.
$$
\end{enumerate}
\end{thm}
From this theorem, we see that the contact triad connection $\nabla$ canonically induces
a Hermitian connection $\nabla^\pi$ for the Hermitian vector bundle $(\xi, J, g_\xi)$, and we call it the \emph{contact Hermitian connection}. This connection will be used to study estimates for the $\pi$-energy in later sections.

Moreover, the following fundamental properties of the contact triad connection was
proved in \cite{oh-wang1}, which will be useful to perform tensorial calculations later.

\begin{cor}\label{cor:connection}
Let $\nabla$ be the contact triad connection. Then
\begin{enumerate}
\item For any vector field $Y$ on $M$,
\be\label{eq:nablaYX}
\nabla_Y R_\lambda = \frac{1}{2}(\CL_{R_\lambda}J)JY;
\ee
\item $\lambda(T|_\xi)=d\lambda$.
\end{enumerate}
\end{cor}

We refer readers to \cite{oh-wang1} for more discussion on the contact triad connection and its relation with other related canonical type connections.

\subsection{$\epsilon$-regularity and interior density estimate}

Now we establish the elliptic a priori estimates for the boundary
value problem \eqref{eq:contacton-Legendrian-bdy-intro} which extends the interior estimates proved
in \cite{oh-wang2}.

The following $\epsilon$-regularity and interior density
estimate was proved in \cite[Corollary 5.2]{oh-wang2}.

\begin{thm}\label{thm:density}
There exist constants $C, \, \epsilon_0$ and $r_0 > 0$, depending only on $J$ and
the Hermitian metric $h$ on $\dot \Sigma$, such that for any
 $C^1$ contact instanton $w: \dot \Sigma \to M$ with
$$
E(r_0): = \frac{1}{2}\int_{D(r_0)} |dw|^2 \leq \epsilon_0,
$$
and discs $D(2r) \subset \operatorname{Int}\Sigma$ with $0 < 2r \leq r_0$,
$w$ satisfies
\be\label{eq:schoen's}
\max_{\sigma \in (0,r]} \left(\sigma^2 \sup_{D(r-\sigma)}
e(w)\right) \leq CE(r)
\ee
for all $0< r \leq r_0$. In particular, letting $\sigma = r/2$, we obtain
\be\label{eq:supeu}
\sup_{D(r/2)} |dw|^2 \leq \frac{4C E(r)}{r^2}
\ee
for all $r \leq r_0$.
\end{thm}

The proof of this theorem is a consequence of the following differential
inequality proved therein
$$
\Delta e(w)\leq Ce(w)^2+\|K\|_{L^\infty(\dot\Sigma)}e(w),
$$
with
$$
C=2\|\CL_{R_\lambda}J\|^2_{C^0(M)}+\|\nabla^\pi(\CL_{R_\lambda}J)\|_{C^0(M)}+\|\text{\rm Ric}\|_{C^0(M)}+1
$$
which is a positive constant independent of $w$. (See \cite[Theorem 5.1]{oh-wang2}.)
Next we prove the local a priori estimates.

\subsection{Local coercive $W^{2,2}$ estimate}

The following pointwise inequality is derived in \cite{oh-wang2}.

\begin{lem}[Equation (5.13), \cite{oh-wang2}]\label{lem:second:derivative}
Let $w$ be any contact instanton $w: (\dot \Sigma,j) \to (M; \lambda, J)$ i.e., any
map satisfying the equation
\be\label{eq:contacton}
\delbar^\pi w = 0 , \quad d(w^*\lambda \circ j) = 0.
\ee
Then we have
\be\label{eq:second-derivative}
|\nabla(dw)|^2 \leq C_1 |dw|^4 - 4K |dw|^2 - 2\Delta |dw|^2
\ee
\end{lem}
We recall $|dw|^2 = |d^\pi w|^2 + |w^*\lambda|^2 = |\del^\pi w|^2 + |w^*\lambda|^2$
for the contact instanton $w$ which satisfies $\delbar^\pi w = 0$.

We now establish the following a priori estimate

\begin{thm}\label{thm:local-W12} Let $w: \R \times [0,1] \to M$ satisfy \eqref{eq:contacton-Legendrian-bdy-intro}.
Then for any relatively compact domains $D_1$ and $D_2$ in
$\dot\Sigma$ such that $\overline{D_1}\subset D_2$, we have
$$
\|dw\|^2_{W^{1,2}(D_1)}\leq C_1 \|dw\|^2_{L^2(D_2)} + C_2 \|dw\|^4_{L^4(D_2)} + C_3 \|dw\|^3_{L^3(\del D_2)}
$$
where $C_1, \ C_2$ are some constants which
depend only on $D_1$, $D_2$ and $(M,\lambda, J)$ and $C_3$ is a
constant which also depends on $R_i$ with $w(\del D_2) \subset R_i$ as well.
\end{thm}
\begin{proof} The proof is similar to that of \cite[Proposition 5.3]{oh-wang2} given in
\cite[Appendix C]{oh-wang2} and so we just indicate necessary modification needed to
handle the Legendrian boundary condition in the estimates.

We have only to consider the case of a pair of semi-discs $D_1,\, D_2 \subset \dot \Sigma$
with $\overline D_1 \subset D_2$ such that $\del D_2 \subset \del D_1 \subset \del \dot \Sigma$.
(The open disc cases of $D_1 \subset D_2$ are already treated in \cite[Appendix C]{oh-wang2}.)

For the pair of given domains $D_1$ and $D_2$, we choose another domain $D$ such that
$\overline D_1 \subset D \subset \overline D \subset D_2$ and a smooth cut-off function $\chi:D_2\to \R$ such that
$\chi\geq 0$ and
$\chi\equiv 1$ on $\overline{D_1}$, $\chi\equiv 0$ on $D_2-D$.
Multiplying  \eqref{eq:second-derivative} by $\chi^2$ and integrating over $D$, we get
\beastar
\int_{D_1}|\nabla(dw)|^2&\leq&\int_{D}\chi^2|\nabla(dw)|^2\\
&\leq&C_1\int_{D}\chi^2|dw|^4-4\int_{D}K\chi^2|dw|^2-2\int_{D}\chi^2\Delta e\\
&\leq&C_1\int_{D_2}|dw|^4+4\|K\|_{L^\infty(\dot\Sigma)}\int_{D_2}|dw|^2-2\int_{D}\chi^2\Delta e
\eeastar
where $C_1$ is the same constant as the one appearing in \eqref{eq:second-derivative}.

We now deal with the last term $\int_{D_2}\chi^2 \Delta e$.
We rewrite
\beastar
\chi^2\Delta e\, dA&=&*(\chi^2 \Delta e)=\chi^2 *\Delta e
=-\chi^2 d*de\\
&=&-d(\chi^2 *de)+2\chi d\chi\wedge (*de).
\eeastar
By the same argument as in \cite[p.677]{oh-wang2}, we get the estimate
\beastar
\left|\int_{D}\chi d\chi\wedge(*de)\right| \leq \frac{1}{\epsilon}\int_{D}\chi^2|\nabla(dw)|^2\,dA+\epsilon\|d\chi\|_{C^0(D)}^2\int_{D}|dw|^2\,dA
\eeastar
for any $\epsilon > 0$.

We now estimate the integral
$$
\int_D -d(\chi^2 * de) = \int_{\del D} - \chi^2 *de
$$
by Stokes' formula. \emph{This is where the current
estimate deviates from that of \cite[p.677]{oh-wang2}.}
\begin{lem}\label{lem:Neunman-bdy} $w$ satisfies the Neunman boundary condition,
i.e., $\frac{\del w}{\del t} \perp TR_i$.
\end{lem}
\begin{proof} Since $R_i$ are Legendrian, we have
$$
\frac{\del w}{\del \tau} = \left(\frac{\del w}{\del \tau}\right)^\pi.
$$
Since $\delbar^\pi w= 0$, we have
$$
- J\left(\frac{\del w}{\del t}\right)^\pi = \left(\frac{\del w}{\del \tau}\right)^\pi = \frac{\del w}{\del \tau} \in TR_i.
$$
Therefore $\left(\frac{\del w}{\del t}\right)^\pi \in NR_i$. Since $R_\lambda \in NR_i$,
this proves
$$
\frac{\del w}{\del t} = \left(\frac{\del w}{\del t}\right)^\pi + \lambda\left(\frac{\del w}{\del t} \right) R_\lambda
$$
is contained in $NR_i$.
\end{proof}

\begin{rem}
We recall that contact triad connection preserves the metric but may have
nonzero torsion, i.e., is not the Levi-Civita connection of the triad metric.
The definition of the second fundamental form of a submanifold $S \subset (M,g)$ for
such a Riemannian connection is still a bilinear map
$
B: TS \times TS \to NS
$
defined by the symmetric average
\be\label{eq:B}
B(X_1,X_2) = \frac12((\nabla_{X_1} X_2)^\perp + (\nabla_{X_2}X_1)^\perp).
\ee
\end{rem}

Since
$$
\frac{\del w}{\del \tau} = \left(\frac{\del w}{\del \tau}\right)^\pi = - J \left(\frac{\del w}{\del t}\right)^\pi,
$$
we have
$$
e := \left|\frac{\del w}{\del \tau}\right|^2 + \left|\frac{\del w}{\del t}\right|^2
= 2 \left|\frac{\del w}{\del \tau}\right|^2 + \left|\lambda\left(\frac{\del w}{\del t}\right)\right|^2.
$$
We then compute
\be\label{eq:*de}
*de|_{\del D} = -\frac{de}{dt} = -4 \left\langle \nabla_t \frac{\del w}{\del \tau},\frac{\del w}{\del \tau} \right\rangle
- 2 \frac{\del}{\del t}\left(\lambda\left(\frac{\del w}{\del t}\right)\right) \cdot \lambda\left(\frac{\del w}{\del t}\right).
\ee
Since $w$ satisfies Neunman boundary condition,
the first term becomes
\be\label{eq:-B}
- 4\left \langle B\left(\frac{\del w}{\del \tau},\frac{\del w}{\del \tau}\right),\frac{\del w}{\del t}\right \rangle
\ee
with the second fundament form \eqref{eq:B}.
For the second, we compute
$$
\frac{\del}{\del t}\left(\lambda\left(\frac{\del w}{\del t}\right)\right)
= \nabla_t \lambda\left(\frac{\del w}{\del t}\right) + \lambda\left(\nabla_t \frac{\del w}{\del t}\right).
$$
We write
$$
\frac{\del w}{\del t} = \left(\frac{\del w}{\del t}\right)^\pi + \lambda\left(\frac{\del w}{\del t}\right)R_\lambda
$$
and then
$$
\nabla_t \frac{\del w}{\del t} = \nabla_{\left(\frac{\del w}{\del t}\right)^\pi}
\left(\frac{\del w}{\del t}\right) +
\lambda\left(\frac{\del w}{\del t}\right) \nabla_{R_\lambda} \frac{\del w}{\del t}.
$$
Using the torsion property and $\nabla_{R_\lambda}X \in \xi$ for all $X$ (see Theorem \ref{thm:connection}
(2) and (3) respectively),
we compute
$$
\lambda\left(\nabla_{R_\lambda}\frac{\del w}{\del t}\right) =
\lambda\left(\nabla_{\frac{\del w}{\del t}}R_\lambda\right) = 0.
$$
Therefore using these vanishing and the torsion properties of $\nabla$ again, we derive
\beastar
\lambda\left(\nabla_t \frac{\del w}{\del t}\right) & = &
\lambda\left(\nabla_{\left(\frac{\del w}{\del t}\right)^\pi} \left(\frac{\del w}{\del t}\right)\right)\\
& = & \lambda\left(\nabla_{\frac{\del w}{\del t}}\left(\frac{\del w}{\del t}\right)^\pi
+ T\left(\left(\frac{\del w}{\del t}\right)^\pi,\frac{\del w}{\del t}\right)\right)
=\lambda\left(\nabla_{\frac{\del w}{\del t}}\left(\frac{\del w}{\del t}\right)^\pi\right)\\
& = & \lambda\left(\nabla_{\left(\frac{\del w}{\del t}\right)^\pi}\left(\frac{\del w}{\del t}\right)^\pi\right)
+ \lambda\left(\lambda\left(\frac{\del w}{\del t}\right) \nabla_{R_\lambda}\left(\frac{\del w}{\del t}\right)^\pi
\right)\\
& = &
\lambda\left(\nabla_{\left(\frac{\del w}{\del t}\right)^\pi} \left(\frac{\del w}{\del t}\right)^\pi\right)
= - \lambda\left(\nabla_{\left(\frac{\del w}{\del \tau}\right)^\pi} \left(\frac{\del w}{\del \tau}\right)^\pi\right)\\
& = & \lambda \left(B\left(\left(\frac{\del w}{\del \tau}\right)^\pi, \left(\frac{\del w}{\del \tau}\right)^\pi\right)\right)
\eeastar
on $\del D$. Here  the penultimate equality follows from the following lemma
\begin{lem}
$$
\lambda\left(\nabla_{\left(\frac{\del w}{\del t}\right)^\pi} \left(\frac{\del w}{\del t}\right)^\pi\right)
= - \lambda\left(\nabla_{\left(\frac{\del w}{\del \tau}\right)^\pi} \left(\frac{\del w}{\del \tau}\right)^\pi\right).
 $$
 \end{lem}
 \begin{proof} From the equation
$$
(\delbar w)^\pi\left(\frac{\del}{\del \tau}\right) = 0
$$
we get $\frac{\del w}{\del t} = J \frac{\del w}{\del \tau}$. Therefore
\beastar
\lambda\left(\nabla_{\left(\frac{\del w}{\del t}\right)^\pi} \left(\frac{\del w}{\del t}\right)^\pi\right)
& = &\lambda\left(\nabla_{J\left(\frac{\del w}{\del \tau}\right)^\pi} J \left(\frac{\del w}{\del \tau}\right)^\pi\right)
= \lambda\left(J \nabla_{J\left(\frac{\del w}{\del \tau}\right)^\pi} \left(\frac{\del w}{\del \tau}\right)^\pi\right)\\
& = & \lambda\left(J\left(\nabla_{\left(\frac{\del w}{\del \tau}\right)^\pi} J\left(\frac{\del w}{\del \tau}\right)^\pi
+  T\left(J\left(\frac{\del w}{\del \tau}\right)^\pi,\left(\frac{\del w}{\del \tau}\right)^\pi\right)\right)\right)\\
& = & \lambda\left(J^2 \nabla_{\left(\frac{\del w}{\del \tau}\right)^\pi} \left(\frac{\del w}{\del \tau}\right)^\pi\right)
= -\lambda\left(\nabla_{\left(\frac{\del w}{\del \tau}\right)^\pi} \left(\frac{\del w}{\del \tau}\right)^\pi\right)
\eeastar
where we repeatedly use $\nabla_Y J = 0$ for any $Y \in \xi$ and also use
 the torsion property \eqref{eq:TJYYxi} for the penultimate equality.
This finishes the proof.
\end{proof}

This proves
$$
\left|\lambda\left(\nabla_t \frac{\del w}{\del t}\right)\right|
\leq \|\lambda\|_{C^0}\|B\|_{C^0}\left|\frac{\del w}{\del t}\right|^2.
$$
On the other hand, we compute
$$
\left|\nabla_t \lambda\left(\frac{\del w}{\del t}\right)\right| \leq \|\lambda\|_{C^1}\left|\frac{\del w}{\del t}\right|^2.
$$

By substituting these and \eqref{eq:-B} into \eqref{eq:*de}, we obtain
\begin{lem}
\beastar
|*de|_{\del D}| & \leq & 4 \left|\left \langle B\left(\frac{\del w}{\del \tau},\frac{\del w}{\del \tau}\right),\frac{\del w}{\del t}\right \rangle\right| + \|\lambda\|_{C^0}^2\|B\|_{C^0}\left|\frac{\del w}{\del t}\right|^3
+ \|\lambda\|_{C^1}\|\lambda\|_{C^0} \left|\frac{\del w}{\del t}\right|^3 \\
& \leq & C_3 |dw|^3_{\del D}
\eeastar
for $C_3: = 4\|B\|_{C^0} +  \|\lambda\|_{C^0}^2\|B\|_{C^0} + \|\lambda\|_{C^1}\|\lambda\|_{C^0}$ where
$B = B_i$ is the second fundamental form of $R_i$.
\end{lem}

Then we can sum all the estimates above and get
\beastar
\int_{D}\chi^2|\nabla(dw)|^2
&\leq& \int_D\frac{2\chi^2}{\epsilon}|\nabla(dw)|^2\\
&{}&+\left(4\|K\|_{L^\infty(\dot\Sigma)}+2\|d\chi\|_{C^0(D)}\epsilon\right)\int_{D_2}|dw|^2\\
&{}&+C_1\int_{D_2}|dw|^4 +  C_3 \int_{\del D} |dw|^3
\eeastar
We take $\epsilon=4$. Then
\beastar
&{}&\int_{D_1}|\nabla(dw)|^2\leq\int_{D}\chi^2|\nabla(dw)|^2\nonumber\\
&\leq&\left(8\|K\|_{L^\infty(\dot\Sigma)}+16\|d\chi\|^2_{C^0(D)}\right)\int_{D_2}|dw|^2
+2C_1\int_{D_2}|dw|^4 + C_3 \int_{\del D}|dw|^3 .\nonumber\\
\eeastar
By setting
\beastar
C_1(D_1,D_2) & = & 8\|K\|_{L^\infty(\dot\Sigma)}+16\|d\chi\|^2_{C^0(D)}\\
C_3(D_1,D_3) & = & 4\|B\|_{C^0} +  \|\lambda\|_{C^0}^2\|B\|_{C^0} + \|\lambda\|_{C^1}\|\lambda\|_{C^0}
\eeastar
and $C_2 = C_2(D_1,D_2) = 2C_1$ the constant given in \eqref{eq:second-derivative},
we have finished the proof.
\end{proof}

\section{$C^{k,\alpha}$ coercive estimates for $k \geq 1$}
\label{sec:Wk+22-estimates}

Once we have established $W^{2,2}$ estimate, we could proceed with the $W^{k+2,2}$ estimate $k \geq 1$
inductively as in \cite[Section 5.2]{oh-wang2}.
The effect of the Legendrian boundary condition on
the higher derivative estimate is not quite straightforward, although it should be doable.
Instead we take an easier path by expressing the following fundamental equation in
the isothermal coordinates of $(\dot \Sigma,j)$.

\begin{thm}[Theorem 4.2, \cite{oh-wang1}]
Let $w$ satisfy $\delbar^\pi w=0$. Then
\be\label{eq:fundamental}
d^{\nabla^\pi}(d^\pi w) = -w^*\lambda\circ j \wedge\left( \frac{1}{2}(\CL_{R_\lambda}J)\, d^\pi w\right).
\ee
\end{thm}

In an isothermal coordinates $z = x+iy$, with the evaluation of $(\frac{\del}{\del x},\frac{\del}{\del y})$
into \eqref{eq:fundamental}, the equation becomes
$$
\nabla_x^\pi \zeta + J \nabla_y^\pi \zeta
+ \frac{1}{2} \lambda\left(\frac{\del w}{\del y}\right)(\CL_{R_\lambda}J)\zeta - \frac{1}{2}
\lambda\left(\frac{\del w}{\del x}\right)(\CL_{R_\lambda}J)J\zeta =0.
$$
(See \cite[Corollary 4.3]{oh-wang1}.)

We note that
by the Sobolev embedding, $ W^{2,2} \subset C^{0,\alpha}$ for $0 \leq \alpha < 1/2$.
Therefore we start from $C^{0,\alpha}$ bound with $0 < \alpha <1/2$ and will inductively bootstrap it to
$C^{k, \alpha}$ bounds for $k \geq 1$ by a Schauder-type estimates, instead of the $W^{k+2,2}$-estimates.
The current approach also simplifies the higher derivative estimate given in \cite[Section 5.2]{oh-wang2}
which do the $W^{k+2,2}$-estimates instead.

We recall that $(\dot \Sigma,j)$ is equipped with strip-like coordinates
near a punctured neighborhood $U_i\setminus \{z_i\}$ for each $z_i$.
We also equip a K\"ahler metric $h$ on $\dot \Sigma$ that is strip-like, i.e.,
$h = d\tau^2 + dt^2$ thereon.

\begin{thm}\label{thm:local-regularity} Let $w$ be a contact instanton satisfying \eqref{eq:contacton-Legendrian-bdy-intro}.
Then for any pair of domains $D_1 \subset D_2 \subset \dot \Sigma$ such that $\overline{D_1}\subset D_2$, we have
$$
\|dw\|_{C^{k,\alpha}(D_1)} \leq C \|dw\|_{W^{1,2}(D_2)}
$$
for some constant $C > 0$ depending on $J$, $\lambda$ and $D_1, \, D_2$ but independent of $w$.
\end{thm}
\begin{proof}
WLOG, we assume that $D_2 \subset \dot \Sigma$ is a semi-disc with $\del D \subset \del \dot \Sigma$
and equipped with an isothermal coordinates $(x,y)$ such that
$$
D_2 = \{ (x,y) \mid |x|^2 + |y|^2 < \delta, \, y \geq 0\}
$$
for some $\delta > 0$
and so $\del D_2 \subset \{(x,y) \in D \mid y = 0\}$. Assume $D_1 \subset D_2$
is the semi-disc with radius $\delta /2$.
We denote $\zeta = \pi \frac{\del w}{\del x}$, $\eta = \pi \frac{\del w}{\del y}$
as in \cite{oh-wang2}, and consider the complex-valued function
$$
\alpha(x,y) = \lambda\left(\frac{\del w}{\del y}\right)
+ \sqrt{-1}\left(\lambda\left(\frac{\del w}{\del x}\right)\right)
$$
as in \cite[Subsection 11.5]{oh-wang3}. We note that since $w$ satisfies the
Legendrian boundary condition, we have
\be\label{eq:lambda(delw)=0}
\lambda\left(\frac{\del w}{\del x}\right) = 0
\ee
on $\del D_2$.

\begin{lem}[Lemma 11.19 \cite{oh-wang3}]\label{lem:*dw*lambda}
 Let $\zeta = \pi \frac{\del w}{\del x}$. Then
$$
*d(w^*\lambda) =|\zeta|^2.
$$
\end{lem}
Combining Lemma \ref{lem:*dw*lambda} together with the equation $d(w^*\lambda\circ j)=0$, we
notice that $\alpha$ satisfies the equations
\be\label{eq:atatau-equation}
\begin{cases}
\delbar \alpha =\nu, \quad \nu
=\frac{1}{2}|\zeta|^2 + \sqrt{-1}\cdot 0\\
\alpha(z) \in \R \quad z \in \del D_2
\end{cases}
\ee
thanks to \eqref{eq:lambda(delw)=0},
where $\delbar=\frac{1}{2}\left(\frac{\del}{\del x}+\sqrt{-1}\frac{\del}{\del y}\right)$
is the standard Cauchy-Riemann operator for the standard complex structure $J_0=\sqrt{-1}$.

Then we arrive at the following system of equations for the pair $(\zeta,\alpha)$
\be\label{eq:main-eq-isothermal}
\begin{cases}\nabla_x^\pi \zeta + J \nabla_y^\pi \zeta
+ \frac{1}{2} \lambda(\frac{\del w}{\del y})(\CL_{R_\lambda}J)\zeta - \frac{1}{2} \lambda(\frac{\del w}{\del x})(\CL_{R_\lambda}J)J\zeta =0\\
\zeta(z) \in TR_i \quad \text{for } \, z \in \del D_2
\end{cases}
\ee
for some $i = 0, \ldots, k$, and
\be\label{eq:equation-for-alpha}
\begin{cases}
\delbar \alpha = \frac{1}{2}|\zeta|^2 \\
\alpha(z) \in \R \quad \text{for } \, z \in \del D_2.
\end{cases}
\ee

These two equations form a nonlinear elliptic system for $(\zeta,\alpha)$ which are coupled: $\alpha$ is fed into
\eqref{eq:main-eq-isothermal} through its coefficients and then $\zeta$ provides the input
for the equation \eqref{eq:equation-for-alpha} and then back and forth. Using this structure of
coupling, we obtain the higher derivative estimates by alternating boot strap arguments between $\zeta$ and $\alpha$
which is now in order.

It is obvious to see that \eqref{eq:main-eq-isothermal} is a linear elliptic equation for $\zeta$ with
$W^{1,2}$ coefficients by the above $W^{2,2}$-estimate for $w$ where we have
\beastar
\frac{1}{2} \lambda(\frac{\del w}{\del y})(\CL_{R_\lambda}J)
& = & \frac{1}{2} (\text{\rm Re} \,\alpha +T) (\CL_{R_\lambda}J)\\
\frac{1}{2} \lambda(\frac{\del w}{\del x})(\CL_{R_\lambda}J)
& = & \frac{1}{2} (\text{\rm Im}\, \alpha) (\CL_{R_\lambda}J)J.
\eeastar
On the other hand, by the standard estimate for the Riemann-Hilbert problem
with the real (or imaginary) boundary condition, it follows from
\eqref{eq:equation-for-alpha} that $\alpha$ is in fact in $W^{2,2}$ and so
is in $C^{0,\delta}$ with say $\delta = 2/5 < 1/2$. By the standard Schauder estimate
applied to \eqref{eq:main-eq-isothermal}, $\eta$ is indeed in $C^{1+\delta}$.
By substituting $\eta$ back into \eqref{eq:equation-for-alpha}, we get $\alpha$ in
$C^{1+\delta}$. Repeating this alternating process between \eqref{eq:main-eq-isothermal}
and \eqref{eq:equation-for-alpha}, we have established the $C^k$-estimate for all $k \geq 1$.
This finishes the proof.
\end{proof}

\section{Vanishing of asymptotic charge and subsequence convergence}
\label{sec:subsequence-convergence}

In this section, we study the asymptotic behavior of contact instantons
on the Riemann surface $(\dot\Sigma, j)$ associated with a metric $h$ with \emph{strip-like ends}.
To be precise, we assume there exists a compact set $K_\Sigma\subset \dot\Sigma$,
such that $\dot\Sigma-\Int(K_\Sigma)$ is a disjoint union of punctured semi-disks
 each of which is isometric to the half strip $[0, \infty)\times [0,1]$ or $(-\infty, 0]\times [0,1]$, where
the choice of positive or negative strips depends on the choice of analytic coordinates
at the punctures.
We denote by $\{p^+_i\}_{i=1, \cdots, l^+}$ the positive punctures, and by $\{p^-_j\}_{j=1, \cdots, l^-}$ the negative punctures.
Here $l=l^++l^-$. Denote by $\phi^{\pm}_i$ such strip-like coordinates.
We first state our assumptions for the study of the behavior of boundary punctures.
(The case of interior punctures is treated in \cite[Section 6]{oh-wang2}.)

\begin{defn}Let $\dot\Sigma$ be a boundary-punctured Riemann surface of genus zero with punctures
$\{p^+_i\}_{i=1, \cdots, l^+}\cup \{p^-_j\}_{j=1, \cdots, l^-}$ equipped
with a metric $h$ with \emph{strip-like ends} outside a compact subset $K_\Sigma$.
Let
$w: \dot \Sigma \to M$ be any smooth map with Legendrian boundary condition.
We define the total $\pi$-harmonic energy $E^\pi(w)$
by
\be\label{eq:endenergy}
E^\pi(w) = E^\pi_{(\lambda,J;\dot\Sigma,h)}(w) = \frac{1}{2} \int_{\dot \Sigma} |d^\pi w|^2
\ee
where the norm is taken in terms of the given metric $h$ on $\dot \Sigma$ and the triad metric on $M$.
\end{defn}

We put the following hypotheses in our asymptotic study of the finite
energy contact instanton maps $w$ as in \cite{oh-wang2}, \emph{except not requiring the charge vanishing
condition $Q = 0$, which itself we will prove here under the hypothesis using the Legendrian boundary condition}:

\begin{hypo}\label{hypo:basic}
Let $h$ be the metric on $\dot \Sigma$ given above.
Assume $w:\dot\Sigma\to M$ satisfies the contact instanton equation \eqref{eq:contacton-Legendrian-bdy-intro}
and
\begin{enumerate}
\item $E^\pi_{(\lambda,J;\dot\Sigma,h)}(w)<\infty$ (finite $\pi$-energy);
\item $\|d w\|_{C^0(\dot\Sigma)} <\infty$.
\item $\Image w \subset K \subset M$ for some compact set $K$.
\end{enumerate}
\end{hypo}

Throughout this section, we work locally near one boundary puncture $p$, i.e., on a punctured semi-disc
$D^\delta(p) \setminus \{p\}$. By taking the associated conformal coordinates $\phi^+ = (\tau,t)
:D^\delta(p) \setminus \{p\} \to [0, \infty)\times [0,1] \to $ such that $h = d\tau^2 + dt^2$,
we need only look at a map $w$ defined on the half strip $[0, \infty)\times [0,1]\to M$
without loss of generality.

The above finite $\pi$-energy and $C^0$ bound hypotheses imply
\be\label{eq:hypo-basic-pt}
\int_{[0, \infty)\times [0,1]}|d^\pi w|^2 \, d\tau \, dt <\infty, \quad \|d w\|_{C^0([0, \infty)\times [0,1])}<\infty
\ee
in these coordinates.

Let $w$ satisfy Hypothesis \ref{hypo:basic}. We can associate two
natural asymptotic invariants at each puncture defined as
\bea
T & := & \frac{1}{2}\int_{[0,\infty) \times [0,1]} |d^\pi w|^2 + \int_{\{0\}\times [0,1]}(w|_{\{0\}\times [0,1]})^*\lambda\label{eq:TQ-T}\\
Q & : = & \int_{\{0\}\times [0,1]}((w|_{\{0\}\times [0,1] })^*\lambda\circ j).\label{eq:TQ-Q}
\eea
(Here we only look at positive punctures. The case of negative punctures is similar.)

We will use the following equality which is derived in \cite{oh-wang1}. (See \cite[Equation (3.1)]{oh-wang1}.)
\begin{lem} Let $h$ be a K\"aher metric of $(\Sigma,j)$. Suppose $w$ satisfies $\delbar^\pi w = 0$. Then
\be\label{eq:dw*lambda}
d(w^*\lambda) = \frac12 |d^\pi w|^2 dA
\ee
where $dA$ is the area form of $h$.
\end{lem}

\begin{rem}\label{rem:TQ}
In particular \eqref{eq:dw*lambda} holds for any contact instanton $w$.  By Stokes' formula, we can express
$$
T = \frac{1}{2}\int_{[s,\infty) \times [0,1]} |d^\pi w|^2 + \int_{\{s\}\times [0,1]}(w|_{\{s\}\times [0,1]})^*\lambda, \quad
\text{for any } s\geq 0.
$$

Moreover, since $w$ satisfies $d(w^*\lambda\circ j)=0$ and the Legendrian boundary condition, it follows that the integral
$$
\int_{\{s \}\times [0,1]}(w|_{\{s \}\times [0,1]})^*\lambda\circ j, \quad
\text{for any } s \geq 0
$$
does not depend on $s$ whose common value is nothing but $Q$.
\end{rem}
We call $T$ the \emph{asymptotic contact action}
and $Q$ the \emph{asymptotic contact charge} of the contact instanton $w$ at the given puncture.

The proof of the following subsequence convergence result largely follows that of
\cite[Theorem 6.4]{oh-wang2}. Since we need to take care of the Legendrian boundary condition and
also need to prove that the charge $Q$ vanishes in the course of the proof,
we duplicate the details of the proof therein in the first half of our proof and then
explain in the second half how the charge vanishing occurs
under the Legendrian boundary condition. One may say that \emph{the presence of Legendrian barrier
prevents the instanton from spiraling.}

\begin{thm}[Subsequence Convergence]\label{thm:subsequence}
Let $w:[0, \infty)\times [0,1]\to M$ satisfy the contact instanton equations \eqref{eq:contacton-Legendrian-bdy-intro}
and Hypothesis \eqref{eq:hypo-basic-pt}.

Then for any sequence $s_k\to \infty$, there exists a subsequence, still denoted by $s_k$, and a
massless instanton $w_\infty(\tau,t)$ (i.e., $E^\pi(w_\infty) = 0$)
on the cylinder $\R \times [0,1]$  that satisfies the following:
\begin{enumerate}
\item $\delbar^\pi w_\infty = 0$ and
$$
\lim_{k\to \infty}w(s_k + \tau, t) = w_\infty(\tau,t)
$$
in the $C^l(K \times [0,1], M)$ sense for any $l$, where $K\subset [0,\infty)$ is an arbitrary compact set.
\item $w_\infty$ has vanishing asymptotic charge $Q = 0$ and satisfies $w_\infty(\tau,t)= \gamma(T\, t)$
for some Reeb chord $\gamma$ is some Reeb chord joining $R_0$ and $R_1$ with period $T$ at each puncture.
\item $T > 0$ either for single Legendrian i.e., with the case of one puncture or for a generic nonempty
collection $\{R_0, \cdots, R_k\}$ with $k \geq 1$ with repetition allowed so that $R_i \cap R_j = \emptyset$
for $R_i \neq R_j$.
\end{enumerate}
\end{thm}
\begin{proof}
For a given contact instanton $w: [0, \infty)\times [0,1]\to M$, we consider the family of maps
$w_s: [-s, \infty) \times [0,1] \to M$ defined by
$w_s(\tau, t) = w(\tau + s, t)$ with Legendrian boundary condition
$$
w_s(\tau,0) \in R, \quad w_s(\tau,1) \in R'
$$
for $R,\, R' \in \{R_0, \ldots, R_k\}$.
For any compact set $K\subset \R$, there exists some $s_0$ large enough such that
$K\subset [-s, \infty)$ for every $s\geq s_0$. For such $s\geq s_0$, we can also get an $[s_0, \infty)$-family of maps by defining $w^K_s:=w_s|_{K\times [0,1]}:K\times [0,1]\to M$.

The asymptotic behavior of $w$ at infinity can be understood by studying the limiting behavior of the sequence of maps
$\{w^K_s:K\times [0,1]\to M\}_{s\in [s_0, \infty)}$, for each given compact set $K\subset \R$.

First of all,
it is easy to check that under Hypothesis \ref{hypo:basic}, the family
$\{w^K_s:K\times [0,1]\to M\}_{s\in [s_0, \infty)}$ satisfies the following
\begin{enumerate}
\item
For every $s\in [s_0, \infty)$,
$$
\begin{cases}
\delbar^\pi w^K_s=0, \quad d((w^K_s)^*\lambda\circ j)=0 \\
w_s(\tau,0) \in R, \quad w_s(\tau,1) \in R'.
\end{cases}
$$
\item $\lim_{s\to \infty}\|d^\pi w^K_s\|_{L^2(K\times [0,1])}=0$
\item $\|d w^K_s\|_{C^0(K\times [0,1])}\leq \|d w\|_{C^0([0, \infty)\times [0,1])}<\infty$.
\end{enumerate}

From (1) and (3) together with the compactness of $\overline{\Image w} \subset M$
(which provides a uniform $L^2(K\times [0,1])$ bound)
and Theorem \ref{thm:local-regularity}, we obtain
$$
\|w^K_s\|_{C^{1,\alpha}(K\times [0,1])}\leq C_{K;(1,\alpha)}<\infty,
$$
for some constant $C_{K;(1,\alpha)}$ independent of $s$.
Then
 $\{w^K_s:K\times [0,1]\to M\}_{s\in [s_0, \infty)}$ is sequentially pre-compact.
Therefore, for any sequence $s_k \to \infty$, there exists a subsequence, still denoted by $s_k$,
and some limit $w^K_\infty\in C^{1}(K\times [0,1], M)$ (which may depend on the subsequence $\{s_k\}$), such that
$$
w^K_{s_k}\to w^K_{\infty}, \quad \text {as } k\to \infty,
$$
in the $C^{1}(K\times [0,1], M)$-norm sense.
Furthermore, combining this with (2) and \eqref{eq:dw*lambda}, we get
$$
dw^K_{s_k}\to dw^K_{\infty} \quad \text{and} \quad dw^K_\infty=(w^K_\infty)^*\lambda\otimes R_\lambda,
$$
and both $(w^K_\infty)^*\lambda$ and $(w^K_\infty)^*\lambda\circ j$ are harmonic one-forms: Closedness of
$(w^K_\infty)^*\lambda$ follows from \eqref{eq:dw*lambda} and the convergence
$$
|d^\pi w^K_\infty|^2 = \lim_{k \to \infty}|d^\pi w^K_{s_k}|^2 = 0.
$$
Closedness of $(w^K_\infty)^*\lambda\circ j$, which is equivalent to the coclosedness of $(w^K_\infty)^*\lambda$,
follows from the hypothesis. The limit map $w_\infty^K$ also satisfies
Legendrian boundary conditions
$$
w_\infty^K(\tau,0) \in R, \quad w_\infty^K(\tau,1) \in R'
$$
since $w^K_{s_k} \to w^K_\infty$ in $C^1$-topology and each $w^K_{s_k}$ satisfies
the same Legendrian boundary condition.

Note that these limiting maps $w^K_\infty$ have a common extension $w_\infty: \R\times [0,1]\to M$
by a  diagonal sequence argument where, one takes a sequence of compact sets $K$ one including another and exhausting $\R$.
Then it follows that  $w_\infty$ is $C^1$,  satisfies
$$
\|d w_\infty\|_{C^0(\R\times [0,1])}\leq \|d w\|_{C^0([0, \infty)\times [0,1])}<\infty
$$
and $d^\pi w_\infty = 0$. Therefore
$$
dw_\infty=w_\infty^*\lambda \otimes R_\lambda
$$
Then we derive from Theorem \ref{thm:local-regularity} that $w_\infty$ is actually in $C^\infty$.
Also notice that both $w_\infty^*\lambda$ and $w_\infty^*\lambda\circ j$ are bounded harmonic one-forms on $\R\times [0,1]$.
\begin{lem} We have
$$
w_\infty^*\lambda= b\,dt, \quad w_\infty^*\lambda\circ j=b\,d\tau
$$
for some constant $b$.
\end{lem}
\begin{proof}
We have $w_\infty^*\lambda = f\,d\tau + g \, dt$ for some bounded functions $f, \, g$.
Furthermore $w_\infty^*\lambda$ is harmonic if and only if $(f,g)$ satisfies the Cauchy-Riemann equation
\be\label{eq:harmonic-fg}
\frac{\del f}{\del t} = \frac{\del g}{\del \tau}, \, \frac{\del f}{\del \tau} = - \frac{\del g}{\del t}
\ee
for the complex coordinate $\tau + it$.
This in particular implies $\Delta f = 0$. On the other hand,
the Legendrian boundary condition
$$
w_\infty(\tau,0) \in R, \quad w(\tau,1) \in R'
$$
implies
\be\label{eq:Q=0}
\lambda\left(\frac{\del w_\infty}{\del \tau}(\tau,0)\right) = 0 = \lambda\left(\frac{\del w_\infty}{\del \tau}(\tau,1)\right).
\ee
Therefore $f$ is bounded and satisfies
$$
\begin{cases}
\Delta f = 0 \quad &\text{on }\, R \times [0,1]\\
f(\tau,0) = f(\tau,1) = 0 \quad &\text{for all }\, \tau \in \R.
\end{cases}
$$
It follows from standard results on the harmonic functions on the strip $\R \times [0,1]$
satisfying the Dirichlet boundary condition that $f = 0$.
The equation \eqref{eq:harmonic-fg} in turn implies $\frac{\del g}{\del \tau} = 0 = \frac{\del g}{\del t}$
from which we derive $g$ is a constant function.
Therefore the forms $w_\infty^*\lambda$ and $w_\infty^*\lambda\circ j$  must be written in  the form
$$
w_\infty^*\lambda= b\,dt, \quad w_\infty^*\lambda\circ j=b\,d\tau.
$$
for some constants $b$. This finishes the proof.
\end{proof}

We now show that
$$
Q = 0, \quad T = b.
$$
By taking an arbitrary point $r\in K$, since $w_\infty|_{\{r\}\times [0,1]}$ is the limit of some sequence $w_{s_k}|_{\{r\}\times [0,1]}$
in the $C^1$ sense, we derive
\beastar
b& = & \int_{\{r\}\times [0,1]}(w_\infty|_{\{r\}\times [0,1]})^*\lambda
= \int_{\{r\}\times [0,1]}\lim_{k\to \infty}(w_{s_k}|_{\{r\}\times [0,1]})^*\lambda\\
&=&\lim_{k\to \infty}\int_{\{r\}\times [0,1]}(w_{s_k}|_{\{r\}\times [0,1]})^*\lambda
= \lim_{k\to \infty}\int_{\{r+s_k\}\times [0,1]}(w|_{\{r+s_k\}\times [0,1]})^*\lambda\\
&=&\lim_{k\to \infty}(T-\frac{1}{2}\int_{[r+s_k, \infty)\times [0,1]}|d^\pi w|^2)\\
&=&T-\lim_{k\to \infty}\frac{1}{2}\int_{[r+s_k, \infty)\times [0,1]}|d^\pi w|^2
= T;\\
\\
0 &= & \int_{\{r\}\times [0,1]}(w_\infty|_{\{r\}\times [0,1]})^*\lambda\circ j
= \int_{\{r\}\times [0,1]}\lim_{k\to \infty}(w_{s_k}|_{\{r\}\times [0,1]})^*\lambda\circ j\\
&=&\lim_{k\to \infty}\int_{\{r\}\times [0,1]}(w_{s_k}|_{\{r\}\times [0,1]})^*\lambda\circ j\\
&=&\lim_{k\to \infty}\int_{\{r+s_k\}\times [0,1]}(w|_{\{r+s_k\}\times [0,1]})^*\lambda\circ j
= Q.
\eeastar
Here in the derivation, we use Remark \ref{rem:TQ}.

As in \cite{abbas}, \cite{oh-wang2}, we conclude that the image of $w_\infty$ is contained in
a single leaf of the Reeb foliation by the connectedness of $\R \times [0,1]$.
Let $\gamma: \R \to M$ be a parametrisation of
the leaf so that $\dot \gamma = R_\lambda(\gamma)$. Then we can write $w_\infty(\tau, t)=\gamma(s(\tau, t))$, where
$s:\R\times [0,1]\to \R$ and $s=T\,t+c_0$ since $ds=T\,dt$, where $c_0$ is some constant.

From this we derive that, if $T\neq 0$,
$\gamma$ is a Reeb chord of period $T$ joining $R$ and $R'$.
Of course, if $T$ also vanishes, $w_\infty$ is a constant map and so corresponds to an intersection point $R \cap R'$.

The latter case will not occur for a single Legendrian
and for a generic $k+1$ tuple $\{R_0,\ldots, R_k\}$
since $R_i$'s are pairwise disjoint for generic tuples by the dimensional reason,
because $R_\lambda$ has no zero and so cannot have a constant solution.
This finishes the proof.
\end{proof}

\begin{rem}\label{rem:big-difference} This charge vanishing of $Q = 0$ is a big advantage for
the current open string case which is a big deviation from
the closed string case studied in \cite{oh-wang2,oh-wang3,oh:contacton}.
This vanishing removes the only remaining difficulty, the so called the \emph{appearance of spiraling instantons along the
Reeb core}, towards the Fredholm theory and the compactification of
the moduli space of contact instantons. One could say that the presence of the Legendrian obstacle
blocks this spiraling phenomenon of the contact instantons.
In sequels to the present paper as in \cite{oh:perturbed-contacton-bdy},
we will develop a Fredholm theory and a compactification
in the spirit of \cite{oh:contacton} and apply this compactification for the construction of
the Fukaya type category on contact manifolds whose objects are Legendrian submanifolds
and morphisms and products will be constructed by counting appropriate
instantons with prescribed asymptotic conditions and Legendrian boundary conditions.
\end{rem}

From the previous theorem, we  immediately get the following corollary as in \cite[Section 8]{oh-wang2}.

\begin{cor}\label{cor:tangent-convergence} Let $w: \dot \Sigma \to M$ satisfy the contact instanton equation \eqref{eq:contacton-Legendrian-bdy-intro} and Hypothesis \eqref{eq:hypo-basic-pt}.
Then on each strip-like end with strip-like coordinates $(\tau,t) \in [0,\infty) \times [0,1]$ near a puncture
\beastar
&&\lim_{s\to \infty}\left|\pi \frac{\del w}{\del\tau}(s+\tau, t)\right|=0, \quad
\lim_{s\to \infty}\left|\pi \frac{\del w}{\del t}(s+\tau, t)\right|=0\\
&&\lim_{s\to \infty}\lambda(\frac{\del w}{\del\tau})(s+\tau, t)=0, \quad
\lim_{s\to \infty}\lambda(\frac{\del w}{\del t})(s+\tau, t)=T
\eeastar
and
$$
\lim_{s\to \infty}|\nabla^l dw(s+\tau, t)|=0 \quad \text{for any}\quad l\geq 1.
$$
All the limits are uniform for $(\tau, t)$ in $K\times [0,1]$ with compact $K\subset \R$.
\end{cor}

\section{Exponential $C^\infty$ convergence}
\label{sec:exponential-convergence}

In this section, we improve the subsequence convergence to the exponential $C^\infty$ convergence
under the transversality hypothesis. Suppose that the tuple $R_0, \ldots, R_k$ are
transversal in the sense all pairwise Reeb chords are nondegenerate. In particular we assume
that the tuples are pairwise disjoint.

We closely follow the arguments used in \cite[Section 11]{oh-wang3}
which treats the closed string case \emph{under the hypothesis} of
vanishing charge $Q = 0$. Since in the current open string case, this vanishing is
proved and so all the arguments used therein can be repeated verbatim after adapting them to the presence
of the boundary condition in the arguments. So we will outline the main
arguments of the proofs referring the detailed arguments to \cite[Section 11]{oh-wang3}
but indicate only the necessary changes needed.

\subsection{$L^2$-exponential decay of the Reeb component of $dw$}
\label{subsec:Reeb}

In this subsection, we prove the exponential decay of the Reeb component $w^*\lambda$.
We focus on a punctured neighborhood around a puncture $z_i \in \del \Sigma$ equipped with
strip-like coordinates $(\tau,t) \in [0,\infty) \times [0,1]$.

We again consider a complex-valued function
$$
\alpha(\tau,t) = \left(\lambda(\frac{\del w}{\del t})- T \right) + \sqrt{-1}\left(\lambda(\frac{\del w}{\del\tau})\right).
$$
By the Legendrian boundary condition, we know $\alpha(\tau, i) \in \R$ for $i =0, \, 1$.

The following lemma was proved in the closed string case in \cite{oh-wang3}. For readers' convenience, we
provide some details by indicating how we adapt the argument with the presence of boundary condition.

\begin{lem}[Compare with Lemma 11.20 \cite{oh-wang3}]\label{lem:exp-decay-lemma}
Suppose the complex-valued functions $\alpha$ and $\nu$ defined on $[0, \infty)\times [0,1]$
satisfy
\beastar
\begin{cases}
\delbar \alpha = \nu, \hskip1in \alpha(\tau, i) \in \R \, & \text{\rm for } i = 0,\, 1, \\
\|\nu\|_{L^2([0,1])}+\left\|\nabla\nu\right\|_{L^2([0,1])}\leq Ce^{-\delta \tau}
\, & \text{\rm for some constants } C, \delta>0\\
\lim_{\tau\rightarrow +\infty}\alpha(\tau,t) =0
\end{cases}
\eeastar
then $\|\alpha\|_{L^2(S^1)}\leq \overline{C}e^{-\delta \tau}$
for some constant $\overline{C}$.
\end{lem}
\begin{proof}
Notice that from previous section we have already
established the $W^{1, 2}$-exponential decay of $\nu = \frac{1}{2}|\zeta|^2$.
Once this is established, the proof of this $L^2$-exponential decay result
can be proved again by the standard three-interval method
and so omitted. (See \cite[Theorem 10.11]{oh-wang3} for a general abstract framework
or the Appendix of the arXiv version of \cite{oh-wang2} for friendly details.)
\end{proof}

\subsection{$C^0$ exponential convergence}

Now the $C^0$-exponential convergence of $w(\tau,\cdot)$ to some Reeb chord as $\tau \to \infty$
can be proved from the $L^2$-exponential estimates presented in previous sections by
the same argument as the proof of \cite[Lemma 11.22]{oh-wang3} whose proof is omitted.

\begin{prop}[Compare with Lemma 11.22 \cite{oh-wang3}]\label{prop:czero-convergence}
Under Hypothesis \ref{hypo:basic}, for any contact instanton $w$ satisfying the Legendrian boundary condition,
there exists a unique Reeb orbit $z(\cdot)=\gamma(T\cdot):[0,1] \to M$ with period $T>0$, such that
$$
\|d(w(\tau, \cdot), z(\cdot))
\|_{C^0([0,1])}\rightarrow 0,
$$
as $\tau\rightarrow +\infty$,
where $d$ denotes the distance on $M$ defined by the triad metric.
\end{prop}

We just mention that the proof is based on the following lemma whose proof we
refer readers to that of \cite[Lemma 11.22]{oh-wang3}.
\begin{lem} Let $t \in [0,1]$ be given. Then
for any given $\epsilon > 0$, there exists sufficiently large $\tau_1 > 0$ such that
$$
d(w(\tau,t), w(\tau', t)) < \epsilon
$$
for all $\tau, \, \tau' \geq \tau_1$.
\end{lem}

Then the following $C^0$-exponential convergence is also proved.
\begin{prop}[Compare with Proposition 11.23 \cite{oh-wang3}]
There exist some constants $C>0$, $\delta>0$ and $\tau_0$ large such that for any $\tau>\tau_0$,
\beastar
\|d\left( w(\tau, \cdot), z(\cdot) \right) \|_{C^0([0,1])} &\leq& C\, e^{-\delta \tau}
\eeastar
\end{prop}

\subsection{$C^\infty$-exponential decay of $dw - R_\lambda(w) \, d\tau$}
\label{subsec:Cinftydecaydu}

 So far, we have established the following:
\begin{itemize}
\item $W^{1,2}$-exponential decay of $w$,
\item $C^0$-exponential convergence of $w(\tau,\cdot) \to z(\cdot)$ as $\tau \to \infty$
for some Reeb chord $z$ between two Legendrians $R, \, R'$.
\end{itemize}

Now we are ready to complete the proof of
$C^\infty$-exponential convergence $w(\tau,\cdot) \to z$
by establishing the $C^\infty$-exponential decay of $dw - R_\lambda(w)\, dt$.
The proof of the latter decay is now in order which will be
carried out by the bootstrapping arguments
applied to the system \eqref{eq:contacton-Legendrian-bdy-intro}.

Combining the above three, we have obtained $L^2$-exponential estimates of the full derivative $dw$.
By the bootstrapping argument using the local uniform a priori estimates
on the cylindrical region,
we obtain higher order $C^{k,\alpha}$-exponential decays of the term
$$
\frac{\del w}{\del t} - T R_\lambda(z), \quad \frac{\del w}{\del\tau}
$$
for all $k\geq 0$, where $w(\tau,\cdot)$ converges to $z$ as $\tau \to \infty$ in $C^0$ sense.
This, combined with local elliptic $C^{k,\alpha}$-estimates
given in Theorem \ref{thm:local-regularity}, then completes proof of $C^\infty$-convergence of
$w(\tau,\cdot) \to z$ as $\tau \to \infty$.

\section{The case with one-jet bundles and future works}
\label{sec:future-works}

In the case of one-jet bundles $J^1B$, it is shown in \cite{oh-yso} that the contact Hamilton's equation
$\dot x = X_H(t,x)$ is the critical point equation of the \emph{reduced action functional}
$$
\widetilde \CA_H(\gamma) = \int (\pi \circ \gamma)^*\theta - H(t,\gamma(t))\, dt
$$
defined on the path space
\beastar
&{}& \Omega_0(T^*B,0_{T^*B}) \\
& = & \left\{ \gamma: [0,1] \to T^*B \mid \gamma(0) \in 0_{T^*B}, \, \gamma(1)\,\in Z,
 z(\gamma(t)) = -\int_{[0,t]} \gamma^*(\vec E[H] - H)\right \}
\eeastar

Let $J$ be an $\omega_0$-compatible almost complex structure on $T^*B$ with $\omega_0 = -d\theta$.
This induces a natural $\lambda$-adapted CR-almost complex structure
on $J^1B$ in the sense of \cite{oh-wang2} by pulling it back to
$\xi$ by the isomorphism $\xi \to T(T^*B)$ induced by the restriction to $\xi$ of
the projection $d\pi: T(J^1B) \to T(T^*B)$. Since this class of $CR$-almost complex structures
on $J^1B$ will play an important role for the construction of Legendrian spectral invariants in \cite{oh-yso},
we assign a name to them.

\begin{defn}[Lifted $CR$-almost complex structures] We call a $CR$-almost complex
structure on $J^1B$ a \emph{$T^*B$-lift} if it is lifted to $\xi$ by $d\pi$ above from
$\omega_0$-compatible almost complex structure $J$ on for $\omega_0 = -d\theta$ on $T^*B$.
We denote the associated $CR$-almost complex structure by the same letter $J$ by an abuse of notation.
\end{defn}

Let us first express such a CR-almost complex structure $J$ on $J^1B$
in terms of the coordinate $w = (u,f)$ where $u = \pi \circ w$ and $f = z \circ w$,
we can express $dw = Du + df$ where $Du$ is $\xi$-valued one-form and $df$ is
$\R \langle \frac{\del}{\del z}\rangle$-valued one form respectively.
More specifically we have
$$
Du = (du)^\sharp: T\Sigma \to \xi
$$
is the horizontal lifting of $du \in \Omega^1(u^*T(T^*B))$ to one in $\Omega^1(u^*\xi)$
which induced by the map
$$
\frac{\del}{\del q_i} \to \frac{D}{\del q_i}, \quad \frac{\del}{\del p_i} \to  \frac{\del}{\del p_i}.
$$
Then we have the decomposition
$$
Du = (Du)^{(1,0)} + (Du)^{(0,1)}
$$
with the complex linear and the anti-complex linear part of $Du:(T\Sigma, j) \to(\xi,J)$.
By definition, we have
$$
\delbar^\pi w = (Du)^{(0,1)}, \quad \del^\pi w = (Du)^{(1,0)} =: \delbar^D u.
$$

Let $J$ be a $T^*B$-lifted CR almost complex structure on $J^1B$ and consider the case $\dot \Sigma = \R \times [0,1]$.
In terms of the splitting $w = (u,f)$, we can write \eqref{eq:contacton-Legendrian-bdy-intro} into
$$
\begin{cases}
\delbar^D u-( X_H^\pi \otimes dt)^{(0,1)} = 0,\\
d\left(e^{g_H(w)}(w^*\lambda + w^*H_t \, dt)\circ j\right )= 0,\\
u(\tau,0), \quad u(\tau,1)\in 0_{T^*B}
\end{cases}
$$
with $g_H(w)(z) = g_{(\psi_H^t)^{-1}}(w(z))$ when the pair $(R_0,R_1) = (0_{J^1B},0_{J^1B})$ is considered.
We would like to emphasize that such a splitting $w = (u,f)$
 does not exist on general contact manifolds which is a special feature of
one-jet bundle.

For a $T^*B$-lifted CR-almost complex structure $J$, we can rewrite
into the unperturbed equation in terms of $(v,g)$ similarly as in the cotangent bundle case \cite{oh:jdg,oh:cag}:
\be\label{eq:contacton-in-(v,g)}
\begin{cases}
\delbar^D v = 0, \, d(dg\circ j) = d(v^*\theta \circ j)\\
(v(\tau,0),g(\tau,0)) \in R, \quad v(\tau,1)\in 0_{T^*B}.
\end{cases}
\ee
Using this equation, we will develop a Floer theoretic construction of Legendrian spectral invariants and
their applications in a sequel \cite{oh-yso}.

\appendix

\section{Generic transversality of contact Hamiltonian chords}

We start with the definition

\begin{defn} Let $X_t$ be a contact
vector field and $H = -\lambda(X_t)$ be its contact Hamiltonian. We call a trajectory
$\gamma$ thereof \emph{transversal to $\xi$} if $\lambda(\dot \gamma) \neq 0$.
\end{defn}

\begin{rem}\label{rem:transveral-trajectories}
\begin{enumerate}
\item
Note that any Reeb vector field which corresponds to $H = 1$ is transversal to $\xi$
and the set of transversal trajectories for a given contact vector field are open.
\item However not every contact Hamiltonian trajectory is transversal to $\xi$.
Note that for a general hypersurface, it carries a codimension one hypersurface
along which the contact vector field is tangent to $\xi$ which is
called the dividing set of the hypersurface.
\end{enumerate}
\end{rem}

Now let us given a pair of Legendrian submanifolds $(R_0,R_1)$.
In the rest of the present section, we show that for generic choice of $H$ all Hamiltonian chords
satisfying the \emph{Legendrian boundary condition} for $(R_0,R_1)$ are transversal.

\begin{thm}\label{thm:transversality}
Let $R_0, \, R_1$ be a Legendrian pair.
Then there exists a dense subset of time-dependent function $H = H(t,x) \in C^\infty(\R \times M)$
 as follows:
\begin{itemize}
\item There is a discrete set of times $t_i \in \R$ satisfying $\psi_H^{t_i} \cap R_1 \neq \emptyset$.
\item The associated Hamiltonian chords from $R_0$ to $R_1$ are nondegenerate and transversal to $\xi$.
\end{itemize}
We call any such $H$ a $(R_0,R_1)$-transversal Hamiltonian.
\end{thm}
\begin{proof}

We write $\CH = C^\infty([0,1] \times M)$.
Define a vector bundle $\mathfrak X \to \Omega = \Omega(M;R,R_1)$ by
$$
\mathfrak X = \bigsqcup_{\gamma \in \Omega} \{\gamma\} \times \Gamma(\gamma^*TM)
$$
and write $\mathfrak X_\gamma: = \Gamma(\gamma^*TM)$.

We first note that if $\dot \gamma = X_H(\gamma)$, then
$\lambda(\dot \gamma(t)) = -H(t, \gamma(t))$ and hence $\dot\gamma \in \xi$
is equivalent to $H(t,\gamma(t)) \neq 0$.
Motivated by this, we consider the map
$$
\Phi:\Omega_0(M;R_0) \times \CH \times \R^2 \to \mathfrak X \times \R \times M \times M
$$
defined by
\be\label{eq:Phi}
\Phi(\gamma, H,t_0,t_1) = \left(\dot \gamma - X_H(\gamma), H(t,\gamma(t_0)), \gamma(t_1), \gamma(0)\right)
\ee
\begin{lem} The map $\Phi$ is transversal to the submanifold
$$
0_{\mathfrak X} \times \{0\} \times R_1 \times R_0.
$$
\end{lem}
\begin{proof} The scheme of the proof used to prove this kind of statement is
standard. (See \cite[Section 3]{oh:diagram} for example.) We just mention the reason how
the Hamiltonian chords can be made transversal.
For this purpose, it is enough to show that there is no pair $(t_0,t_1)$ with $0 \leq t_0 \leq t_1$
that satisfy
$$
\begin{cases}
\dot \gamma(t_0) - X_H(\gamma(t_0)) = 0, \quad H(t_0,\gamma(t_0)) = 0\\
\gamma(0) \in R_0, \, \gamma(t_1) \in R_1.
\end{cases}
$$
We have the linearization of $\Phi$ at $(\gamma, H, t_0,t_1) \in \Phi^{-1}(0_{\mathfrak X},0,R_1, R_0)$
$$
D\Phi(\gamma, H, t_0,t_1)(\zeta,g,a_0,a_1)
= \left(\nabla_t\zeta - DX_{H_1}(\gamma(t))(\zeta), g(t_0,\gamma(t_0)), \zeta(a_1),\zeta(0)\right).
$$
From this explicit expression, it is easy to show that $D\Phi(\gamma, H, t_0,t_1)$ is surjective
after taking a suitable completion of relevant functions spaces which is standard and so omit.
(See \cite[Section 3]{oh:diagram} for such details in a similar context for symplectic Hamiltonian
trajectories.)
\end{proof}

Now we combine this generic transversality statement with  a dimension
counting argument as follows.
Thanks to $\dim R_0 + \dim R_1 = \dim M-1$, the index of
the projection map
$$
\Pi: \Phi^{-1}(0_{\mathfrak X} \times \{0\} \times R_1) \to \CH
$$
has index $-1$ and hence $\Pi^{-1}(H) = \emptyset$ for the regular values $H$ of $\Pi$, i.e.,
all chords are transversal to $\xi$.

This finishes the proof of the theorem.
\end{proof}

\end{document}